\documentclass[11pt]{amsart}
\usepackage{amsmath,amsthm}
\usepackage{graphics}
\usepackage{latexsym}
\usepackage{verbatim}
\usepackage{amsmath}
\usepackage{amsthm}
\usepackage{amssymb}
\usepackage{epsfig}
\usepackage{epstopdf}
\usepackage{subcaption}
\usepackage{hyperref}
\usepackage[]{hyperref}
\usepackage{setspace}
\usepackage[]{algpseudocode}
\usepackage{algorithmicx,algorithm}
\hypersetup{urlcolor=blue, citecolor=red}
\usepackage{dsfont}
\usepackage[table]{xcolor}
\usepackage{multirow}
\usepackage{float}
\usepackage{tikz}
\usetikzlibrary{arrows, calc, shapes, positioning, decorations.pathreplacing, intersections,angles,arrows.meta, backgrounds}
\usepackage{comment}
\usepackage{MnSymbol}
\usepackage{diagbox}
\usepackage{multirow}
\usepackage{bbm}
\usepackage{upgreek}

\ifpdf
  \DeclareGraphicsExtensions{.eps,.pdf,.png,.jpg}
\else
  \DeclareGraphicsExtensions{.eps}
\fi
\newtheorem{theorem}{Theorem}[section]
\newtheorem{proposition}[theorem]{Proposition}

\newtheorem{lemma}[theorem]{Lemma}
\newtheorem{remark}[theorem]{Remark}
\newtheorem{definition}[theorem]{Definition}

\newcommand{\RR}{{\mathbb R}}

\newcommand{\rd}{{\rm d}}

\graphicspath{{fig/}}

\title[A new collision avoidance model with random batch resolution strategy]{A new collision avoidance model with random batch resolution strategy} 

\author[Tianlu Chen, Chang Yang, L\'eon Matar Tine and Zhichang Guo]{}

\subjclass[]{}

\email{tianlu\_chen@qq.com}
 \email{yangchang@hit.edu.cn}
 \email{leon-matar.tine@univ-lyon1.fr}
 \email{mathgzc@hit.edu.cn}

\begin{document}
\maketitle

\centerline{\scshape Tianlu Chen}
\medskip
{\footnotesize
    \centerline{School of Mathematics, Harbin Institute of Technology,}
    \centerline{No. 92 West Dazhi Street, Nangang District, 150001 Harbin, China}
}

\medskip

\centerline{\scshape Chang Yang}
\medskip
{\footnotesize
    \centerline{School of Mathematics, Harbin Institute of Technology,}
    \centerline{No. 92 West Dazhi Street, Nangang District, 150001 Harbin, China}
}

\medskip

\centerline{\scshape L\'eon Matar Tine}
\medskip
{\footnotesize
    \centerline{Universit\'e Claude Bernard Lyon 1, Institut Camille Jordan,}
    \centerline{UMR 5208, F-69622, Villeurbanne cedex, France}
}

\medskip

\centerline{\scshape Zhichang Guo}
\medskip
{\footnotesize
    \centerline{School of Mathematics, Harbin Institute of Technology,}
    \centerline{No. 92 West Dazhi Street, Nangang District, 150001 Harbin, China}
}

\begin{abstract}
Research on crowd simulation has important and wide range of applications. The main difficulty  is how to lead all particles with a same and simple rule, especially when particles are numerous. In this  paper, we firstly propose a two dimensional agent-based collision avoidance model, which is a $N$-particles  Newtonian system. The collision interaction force, imminent interaction force and following interaction force  are designed, so that  particles can be guided to their respective destinations without collisions.  The proposed agent-based model is then extended to the corresponding mean field limit model as $N\to\infty$.  Secondly, notice that direct simulation of the $N$-particles Newtonian system is very time-consuming, since the computational complexity is of order $\mathcal{O}(N^2)$. In contrast, we propose an efficient hybrid resolution strategy  to reduce the computational complexity. It is a combination of the Random Batch method (Shi Jin, Lei Li, and Jian-Guo  Liu. Random batch methods (RBM) for interacting particle systems. Journal of Computational Physics,  400:108877, 2020.) and the method based on local particles Newtonian system. Thanks to this hybrid resolution strategy, the computational complexity is reduced to $\mathcal{O}(N)$. Finally, various tests are presented to show robustness and  efficiency of our collision avoidance model and the hybrid  resolution strategy.
\end{abstract}


\tableofcontents
\section{Introduction}
\label{sec:introduction}
Crowd simulation mainly uses computers to simulate the activity of a group of people, such as group walking and targeted crowd behavior, which attracts significant attentions with a wide range of applications \cite{bellomo2011modeling}. It can test the safety for preventing crowd incidents \cite{helbing2007dynamics,helbing2002crowd} or evaluating evacuation scenarios \cite{liu2016agent,BELLOMO20161} and the architecture for assessing the population carrying capacity of public spaces \cite{bah2006agent}. Moreover, the crowd model can be designed for monitoring of natural disasters \cite{mustapha2013modeling} or sensing in rescue missions \cite{kopfstedt2008control,hoekstra2001free}. Motivated by the COVID-19 pandemic, crowd modeling can also be developed to study the complex interactions between crowd motion and virus spreading \cite{shamil2021agent}. A recent review on crowd models can be seen in \cite{bellomo2022towards}.

These important issues motivate further research on the agent-based model for swarming. Inside the swarm, agents with similar size and shape interact with each other, cooperate to accomplish a given task and move to destinations by following the same rules. The nature provides many great examples of decentralized, collective behaviors without colliding in swarms, such as flocks of fishes, groups of birds, swarms of insects \cite{bonabeau1999swarm,camazine2020self,giardina2008collective,degond2017coagulation}. In spite of limited sensing capability and physiological constraints, such biological groups have outstanding performance of maintaining optimized group structure and avoiding collisions. Motivated by research on animals' behavior, many widely used models have been constructed, such as optimal control models \cite{hoogendoorn2003simulation}, traffic following models \cite{lemercier2012realistic} and vision-based models \cite{parzani2017three,ondvrej2010synthetic,degond2013hierarchy}.

In this paper, our interest is to propose an agent-based model for swarming and to give an efficient numerical method to solve the model. The model in this paper belongs to the class of vision-based models, where each agent has a limited sensing capability. To achieve the aim, we inspire the ideas  from vision-based models in \cite{parzani2017three, ondvrej2010synthetic}, where the particles, with limited sensor called ``vision cone'', are considered equal and follow the same simple rules to observe, detect, and react for collision avoidance. It is noted that the model in \cite{parzani2017three} is a Newtonian system constructed for three-dimensional space and each particle only changes its orientation to avoid collisions, maintaining the magnitude of its velocity. Thus unavoidable collision can be caused when particles are dense. Meanwhile, each particle in \cite{ondvrej2010synthetic} is directly manufactured by changing only its velocity in two-dimensional space, and each particle can deviate or decelerate to avoid collisions. In contrast to these two models, our model, called the original model in the sequel, is a two-dimensional Newtonian system with three interacting forces for deviating, decelerating and aligning purposes respectively. The force for decelerating can avoid imminent collisions when particles are too close while the force for aligning can lead particles with the same motion direction to line up, which relieves congestion especially when particles are dense.

To achieve a better understanding of crowd behavior and to increase the reliability of predictions, numerical modeling and simulation are playing an ever-growing role. When the number of particles $N$ is large, it is well known that simulation of the model is very expensive since for each time step, particles interact with each other pairwise and the computational complexity is of order $\mathcal{O}(N^2)$.  In the era of big data, many stochastic algorithms have been proposed to reduce the computational complexity \cite{bottou_1999,bubeck2015convex,ma2015complete}.  In \cite{jin2020random}, the algorithm, called Random Batch Methods (RBM), is developed for interacting particle systems. The idea is that, for a small duration of time, particles are divided into small batches randomly and each particle only interact inside the batch. The size of each batch is prefixed to be $p$ at most ($p\ll N$, often $p=2$). So the computational complexity of RBM is reduced from $\mathcal{O}(N^2)$ to $\mathcal{O}(pN)$. When the  shuffle frequency of batches tends to infinity, the error of RBM with respect to the original model would converge to zero, as is proved in \cite{jin2020random}. 

One can also understand swarming by the mean-field approaches  \cite{lasry2007mean,georges1996dynamical}, where the effect of the surrounding particles is replaced by an averaged one. This substitution principle obtains a mean-field equation from the agent-based model, which is free from the dimensionality. As $N\to\infty$, the mean-field limit about the particles' distribution function is obtained, which satisfies a nonlinear Fokker-Planck equation \cite{jin2020random,jin2022mean}. Moreover, the model solved by RBM can also be seen as a new agent-based model for collision avoidance, called RBM model, and its distribution function converges to the corresponding mean-field limit equation when $N\to\infty$.

Although RBM is efficient for simulation, overlapping of particles can happen if $N$ is large and the  shuffle frequency of batches is not large enough. It is because that neighbouring particles  may not be assigned in the same batch and can not detect each other. This phenomenon is unrealistic and unacceptable in the agent-based model for the purpose to give a motion planning without colliding. To improve the performance of the algorithm, we introduce the Cell-List approach \cite{frenkel2001understanding} to avoid the overlapping phenomenon. The idea of the Cell-List approach is to divide the whole space into cells of equal size and each particle only interact with particles in the same cell or neighbouring cells. Thus, combining the Cell-List approach with RBM, we construct a hybrid method to solve the agent-based model, called the hybrid model in the sequel. Since the volume of each particle is prefixed and the number of particles in one cell has a upper bound independent of the total particles' number $N$, the computational complexity of the hybrid model is thus reduced to $\mathcal{O}(N)$.

The remainder of the paper is organized as follows. Our agent-based model for interacting particles is described in Section~\ref{sec:Themodel}. In Section~\ref{sec:Meanfield}, the associated mean-field limit is formally derived and analyzed. We show the existence of the weak solution and the crucial Lipschitz property of the force field of the mean-field limit under reasonable assumptions. The RBM model and the hybrid model are discussed in Section~\ref{sec:Algorithm}. Section~\ref{sec:Numericalresults} is devoted to numerical experiments of the different models and we show the outstanding performance of the hybrid model compared to the original model and the RBM model. Finally, a conclusion is drawn in Section~\ref{sec:Conclusion}.

\section{Modelling}\label{sec:Modeling}
In this section, our interest is modelling the motion of individuals without colliding with each other or obstacles in two-dimensional space. Following the same framework as in~\cite{parzani2017three},
the model consists of two phases:
\begin{itemize}
\item[(1)] a perception phase, where each particle  will evaluate its potential collision risk with others.
\item[(2)] a decision-making phase, where each particle decides its local optimal motion by virtue of computing its current acceleration, in order to attain the destination without collisions.
\end{itemize}
Therefore, in this section, we will first give a preliminary, which defines some key ingredients for modelling. Then, four different strategies of collision avoidance are considered to define a Newtonian system. Finally, the model is extended to the mean field limit by passing particle number $N\to\infty$.

\subsection{Preliminary}
For clarity, some variables and definitions for describing the motion of  particles and assessing the potential collision risk will be introduced in this part. 

For simplicity, the agents, such as people, are symbolized by $N$ interacting motion particles with position $x_i(t)\in \RR^2$ and velocity $v_i(t)\in \RR^2$ ($i \in \{1,\ldots,N\}$). In the perception phase, we should detect all interactions among the particles. However, not all particles around the particle $i$ would be evaluated because of the limited sensing capability. Therefore, the ``vision cone'' $\mathcal{C}_i$ is introduced to the model to denote the set of particles which can be seen by the particle $i$. The sketch of ``vision cone'' is depicted in Figure~\ref{fig:vcone}, and its definition is given as follows.
\begin{definition}\label{def:visioncone}
Set a threshold number $\kappa\in [-1,1]$. The ``vision cone'' $\mathcal{C}_i$ for the particle $i$ is the cone centered at position $x_i(t)$ with angle $\arccos \kappa$ about the direction $v_i(t)$.
\end{definition}

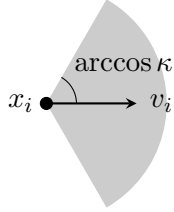
\begin{figure}
\centering
\begin{tikzpicture}[>=stealth, scale=0.8]
\fill[fill=black!20!white] (0,0) -- (-60:2) arc (-60:60:2) --cycle;
\filldraw[fill=black] (0,0) circle (.1) node [left=1pt]{$x_i$};
\draw[->,thick] (0,0) -- (1.5,0) node [right=1pt]{$v_i$};
\draw (0:.5) arc (0:60:.5);
\node at (.2,.2)[label=60:$\arccos \kappa$]{};
\end{tikzpicture}
\caption{\label{fig:vcone}The ``vision cone'' for the particle $i$.}
\end{figure}

When we consider the particle $i$ interacting with the particle $j$, there are totally three different configurations:
\begin{itemize}
\item Safe configuration (depicted in Figure~\ref{fig:confi} (a)): there is no risk for collision to happen. So both particles don't need change their velocities;
\item Blind configuration (depicted in Figure~\ref{fig:confi} (b)): the particle $i$ sees the particle $j$ whereas the particle $j$ does not. Only the particle $i$ is expected to change its velocity to avoid the collision;
\item Unsafe configuration (depicted in Figure~\ref{fig:confi} (c)): both of particles detect each other and will modify their velocities to avoid the collision.
\end{itemize}

\begin{figure}
\begin{center}
\begin{tikzpicture}[>=stealth, scale=0.8]
\tikzset{
vconei/.style = {fill=black!20!white},
vconej/.style = {fill=black!10!white},
dvcone/.style = {fill=black!30!white},
inner sep=0,
minimum size=2mm
}
\begin{scope}
\path [name path=vconexi, fill, vconei] (0,0) node (xi)[label=above:$x_i$,circle, fill=black]{} -- (-50:2) arc (-50:50:2) --cycle coordinate (vconexi);
\path [name path=vconexj, fill, vconej] (-.5,1)node (xj)[label=above:$x_j$,circle, fill=black]{} -- ++(-170:2) arc (-170:-70:2) --cycle coordinate (vconexj);

\draw[->,thick] (xi) -- +(.8,0) node [right=1pt]{$v_i$};
\draw[->,thick] (xj) -- +(-120:.8) node [below=1pt]{$v_j$};

\node at (0,-1.5) {(a)};
\end{scope}

\begin{scope}[xshift=5cm]
\fill [name path=vconexi, vconei] (-2,0) node (xi)[label=above:$x_i$,circle, fill=black]{} -- ++(-50:2) arc (-50:50:2) --cycle coordinate (vconexi);
\fill [name path=vconexj, vconej] (-.5,.5)node (xj)[label=above:$x_j$,circle, fill=black]{} -- ++(-50:2) arc (-50:50:2) --cycle coordinate (vconexj);


\draw[->,thick] (xi) -- +(.8,0) node [right=1pt]{$v_i$};
\draw[->,thick] (xj) -- +(.8,0) node [right=1pt]{$v_j$};

\node at (0,-1.5) {(b)};
\end{scope}

\begin{scope}[xshift=8cm,]
\path [name path=vconexj, fill, vconej] (0,.4)++(20:1.8) node (xj){} -- ++(-170:2) arc (-170:-70:2) --cycle coordinate (vconexj);
\path [name path=vconexi, fill, vconei] (0,.4) node (xi){} -- ++(-50:2) arc (-50:50:2) --cycle coordinate (vconexi);

\draw[->,thick] (xi) -- +(.8,0) node [right=1pt]{$v_i$};
\draw[->,thick] (xj) -- +(-120:.8) node [below right=1pt]{$v_j$};
\node[shape=circle,fill=black,label=above:$x_i$] at (xi){};
\node[shape=circle,fill=black,label=above:$x_j$] at (xj){};

\draw[black!5!white,dashed,thick] (xj) -- ++(-170:2) arc (-170:-70:2) -- cycle;

\node at (1,-1.5) {(c)};
\end{scope}
\end{tikzpicture}
\end{center}
\caption{\label{fig:confi}Three different configurations of the particle $i$ interacting with the particle $j$. (a) Safe configuration, (b) Blind configuration, (c) Unsafe configuration. }
\end{figure}
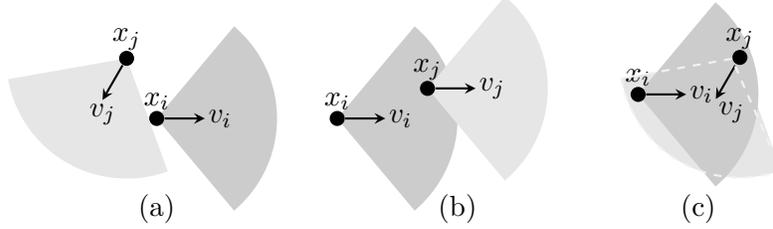

Only blind configuration and unsafe configuration have potential collision risk detected by at least one particle. Therefore, We will focus on these two configurations to design collision avoidance strategies.

Next, we will define three key ingredients in our modelling. Before that let us give some notations.
For any $d\in\mathbb{N}$ and vectors $u=(u_1,\ldots,u_d), v=(v_1,\ldots,v_d) \in \RR^d$, we denote scalar product in $\RR^d$ as follows
\begin{equation*}
<u, v>\  := \sum_{i=1}^d u_iv_i,
\end{equation*}
and the associate norm is denoted by $\|u\|=\sqrt{<u,u>}$. We denote the vector product by $u\times v$, which is defined as the following scalar in the two-dimensional space($d=2$):
\begin{equation*}
u\times v = u_1v_2-u_2v_1.
\end{equation*}
Essentially, it is just cross product in $\RR^3$ restricted to $\RR^2$. And we denote the relative azimuthal angle of $u\in\RR^2$ from $v\in\RR^2$ by $\alpha(u, v)$ (see Figure~\ref{fig:alpha}).
\begin{figure}
\centering
\begin{tikzpicture}[
par/.style={circle,inner sep=0,minimum size=2mm,fill=black!80},
ppar/.style={par,fill=red!80},
>=stealth,scale=0.6]
\node at (0,0) [par]{};
\draw[->,very thick] (0,0) -- (3,0) node [right]{$v$};
\draw[->,very thick] (0,0) -- (120:3) node [above]{$u$};
\draw[shorten >=1pt,>={Stealth[round]},->,thick] (1.,0) arc (0:120:1.);
\node at (0.9,1.3)[]{$\alpha(u, v)$};
\begin{scope}[on background layer]
\fill[black!20](0,0)--(1,0)arc(0:120:1)--cycle;
\end{scope}
\end{tikzpicture}
\caption{\label{fig:alpha}The sketch of the relative azimuthal angle $\alpha(u, v)$.}
\end{figure}
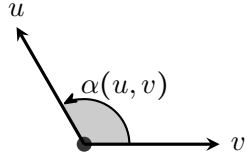

Moreover,  we make two practical assumptions for the modelling as below:
\begin{itemize}
\item[(A1)] Each particle can detect relative positions and velocities of particles in its vision.

\item[(A2)] In the perception phase, each particle makes the assessment based on the assumption that others will maintain a constant velocity without changing motion direction.

\end{itemize}
An interaction between two particles is illustrated in Figure~\ref{fig:sketch}, where the particle $i$ interacts with the particle $j$. 
Let us denote the time $t_0$ as the moment when the particle $i$ is interacting with the particle $j$, and their corresponding positions at $t=t_0$ are abbreviated to $x_i,x_j$ respectively. The interaction points, denoted as $\overline{x}_i,\overline{x}_j$, are depicted in Figure~\ref{fig:sketch}, with the definition given as follows.
\begin{definition}\label{def:sketch}
Under the assumption~(A2), the interaction points $\overline{x}_i,\overline{x}_j$ are the positions for the two particles $i$ and $j$ when their relative distance between each other is minimal, {\it i.e.}
\begin{equation*}
\|\overline{x}_i-\overline{x}_j\| = \min_{t\in \RR} \|x_i(t)-x_j(t)\|.
\end{equation*}
\end{definition}

\begin{figure}
\centering
\begin{tikzpicture}[inner sep=0, minimum size=2mm, scale=0.6]
\draw[dashed] (-5,0) node (xi) [circle, fill=black, label=above:$x_i$]{} -- (0,0) node (bxi) [circle, fill=black, label=120:$\bar{x}_i$]{} -- (5,0);

\draw[->, >=stealth, very thick] (xi) -- +(1.6,0) node[label=above:$v_i$]{};

\path (bxi)++(-20:2) node (bxj) [circle, fill=black, label=-60:$\bar{x}_j$]{};
\draw[dashed] (bxj)++(70:5) node (xj)[circle, fill=black, label=right:$x_j$]{} -- (bxj) -- +(-110:3);
\draw[<->] (bxi) -- node[below=1pt]{$D_{ij}$} (bxj);
\draw[->, very thick, >=stealth] (xj) -- ($(xj)!1.6cm!(bxj)$) node [label=-60:$v_j$]{};
\draw[<->] (xi)+(0,-3mm)-- node [below=1pt]{$\tau_{ij}\cdot \|v_i\|$} ($(bxi)+(0,-3mm)$);
\draw (xi) -- +(0,-5mm)
      (bxi) -- +(0,-5mm);
\end{tikzpicture}
\caption{\label{fig:sketch}Illustration of an interaction between two particles. The particle $i\in \{1,\ldots,N\}$ (with position $x_i\in \RR^2$ and velocity $v_i\in \RR^2$ at time $t_0$) is interacting with the particle $j\in \{1,\ldots,N\}(i\neq j)$ (with position $x_j\in \RR^2$ and velocity $v_j\in \RR^2$ at $t=t_0$). Their shortest distance is $D_{ij}$ when they move to the interaction points $\bar{x}_i,\bar{x}_j$, respectively.}
\end{figure}
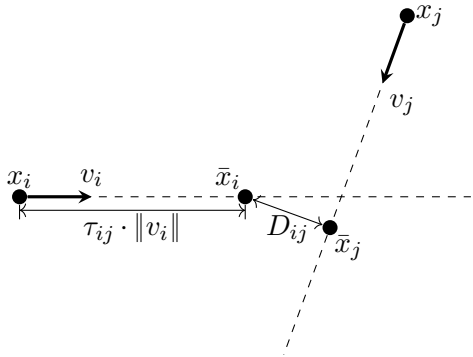

For the convenience of illustrating the model, we introduce two local orthonormal cartesian frames of the particle $i\in\{1,\ldots,N\}$ (depicted in Figure~\ref{fig:frame}). Both frames are centered at position $x_i(t)$. The first frame is denoted by $(e_{\rho_i}, e_{\phi_i})$ with $e_{\rho_i}=v_i/\|v_i\|$; another frame is denoted by $(k_{ij},e_{\alpha_{ij}})$ with $k_{ij}=(x_j-x_i)/\|x_j-x_i\|$. In addition, we denote the relative azimuthal angle from $e_{\rho_{i}}$ to $k_{ij}$ by $\alpha_{ij}\in (-\pi,\pi)$, called the relative bearing angle.

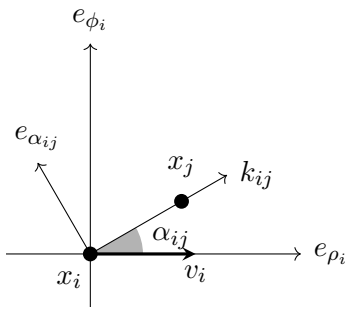
\begin{figure}
\centering
\begin{tikzpicture}[thin,scale=1.4]
\clip (-.8,-0.5) rectangle (2.6,2.6);
\fill[black!30!white] (0,0) -- (.5,0) arc (0:30:.5) -- cycle;
\draw[->] (-2,0)--(2,0) node [right=1pt]{$e_{\rho_i}$};
\draw[->] (0,-2)--(0,2) node [above=1pt]{$e_{\phi_i}$};

\draw[>=stealth,->,very thick] (0,0)--(1,0) node [below] {$v_i$};

\draw[->] (0,0)--(30:1.5) node [right=1pt]{$k_{ij}$};
\draw[->] (0,0)--(120:1) node [above=1pt]{$e_{\alpha_{ij}}$};

\path (.5,0) -- node [right=1pt,yshift=1pt]{$\alpha_{ij}$} (30:.5);
\node (xj) [label=above:{$x_j$},circle] at ($(0,0)!.5!(30:2)$) {};
\fill[black] (0,0) circle (2pt);
\fill[black] (xj) circle (2pt);

\node at (-.2,0)[below=.1] {$x_i$};
\end{tikzpicture}
\caption{\label{fig:frame} Two local orthonormal cartesian frame of the particle $i\in\{1,\ldots,N\}$.}
\end{figure}

The two local orthonormal cartesian frames satisfy the equations below:
\begin{equation*}
\left\{
\begin{array}{l}
k_{ij} = \cos(\alpha_{ij})e_{\rho_i} + \sin(\alpha_{ij})e_{\phi_i},\\[3mm]
e_{\alpha_{ij}} = -\sin(\alpha_{ij})e_{\rho_i} + \cos(\alpha_{ij})e_{\phi_i}.
\end{array}
\right.
\end{equation*}

We denote $d_{ij}(t)$ as the distance between the particle $i$ and the particle $j$, {\it i.e.}
\begin{equation}\label{eq:dij}
d_{ij}(t) = \|x_i(t)-x_j(t)\|.
\end{equation}

Now we are ready to define the three ingredients, which are the time to interaction (TTI), the minimal distance (MD) \cite{moussaid2011simple} and the derivative of the relative bearing angle (DBA) \cite{ondvrej2010synthetic} in following definition.

\begin{definition}\label{def:mdtti}
To evaluate the interaction between the two particles $i$ and $j$, three key ingredients to the model are introduced as follows
\begin{itemize}
\item The time to interaction (TTI), denoted as $\tau_{ij}\in \RR$, represents the time for the particle $i$ to move to its interaction point $\overline{x}_i$ with respect to the particle $j$ from the current position $x_i$ (see Figure~\ref{fig:sketch}).

\item The minimal distance (MD), denoted as $D_{ij}$, represents the closest distance between the two particles $i$ and $j$ predicted under the assumption~(A2) at time $t_0$. From the Definition~\ref{def:sketch}, it satisfies 
\begin{equation*}
D_{ij} = \|\overline{x}_i-\overline{x}_j\|.
\end{equation*}

\item The derivative of the relative bearing angle (DBA), denoted as $\dot\alpha_{ij}$, represents the derivative of the angle between $k_{ij}$ and $e_{\rho_i}$ in the frame constructed, i.e.
\begin{equation*}
\dot\alpha_{ij}={\rd\over\rd t}\alpha(k_{ij},e_{\rho_i})={\rd\over\rd t}\alpha(x_j-x_i,v_i).
\end{equation*}
\end{itemize}
\end{definition}

Similarly as in \cite{parzani2017three}, we state a practical way to compute TTI, MD and DBA,
\begin{itemize}
\item the value of TTI for the particle $i$ with respect to the particle $j$, $\tau_{ij}$, satisfies
\begin{equation*}
\tau_{ij} = -\cfrac{<x_j-x_i,v_j-v_i>}{\|v_j-v_i\|^2},
\end{equation*}
\item the value of MD for the particle $i$ with respect to the particle $j$, $D_{ij}$, satisfies
\begin{equation*}
D_{ij} = \left(\|x_j-x_i\|^2-(\cfrac{<x_j-x_i,v_j-v_i>}{\|v_j-v_i\|})^2\right)^{1\over2},
\end{equation*}
\item the value of DBA for the particle $i$ with respect to the particle $j$, $\dot\alpha_{ij}$, satisfies
\begin{equation*}
\dot\alpha_{ij} = \cfrac{<v_j-v_i,e_{\alpha_{ij}}>}{d_{ij}}.
\end{equation*}
\end{itemize}

\begin{remark}\label{remark:saferadius}
It should be noted that the sign of $\tau_{ij}$ is not restricted to be positive. When $\tau_{ij} < 0$, it means the minimal distance between two particles happened in the past. In this case, there is no potential collision risk. In addition, all particles are identified as a disk with radius $R_0\in\RR^+$. A collision happens when the distance between two particles is smaller than $2R_0$. Thus, we only consider the case when $\tau_{ij}$ is positive and $D_{ij}$ is smaller than the safe radius $R\in\RR^+$. Obviously, we should set $R\geq 2R_0$.
\end{remark}




\subsection{The agent-based collision-avoidance model}\label{sec:Themodel}
Thanks to the above preparation, we can now give the agent-based collision-avoidance model. The key step is to define forces to avoid collisions, here  four different forces will be considered.
\subsubsection{The collision-interaction force}
The collision-interaction force is inspired from~\cite{parzani2017three}, where a force, being perpendicular to the particle velocity, is imposed depending on the potential collision risk. More precisely, two particles are considered to have a potential collision risk if, under the assumption (A2), the MD is small enough and the TTI is nonnegative. Therefore the two particles should take action to change its moving direction to avoid the collision.

In the perception phase, we define the collision-interaction set $\mathcal{K}_i^{Co}(t)$, which consists of particles that have a potential collision risk with the particle $i$ at time $t$, as follows.
\begin{definition}\label{def:K_i}
For a given particle $i$, any particle $j$, satisfying the following three conditions:
\begin{enumerate}
\item $x_j\in \mathcal{C}_i$;
\item $\tau_{ij}$ is nonnegative;
\item $D_{ij}$ is smaller than $R$;
\end{enumerate}
 is in the collision-interaction set $\mathcal{K}_i^{Co}(t)$.
In another word,
\begin{equation}\label{eq:KCo}
\mathcal{K}_i^{Co}(t) = \{j:x_j\in\mathcal{C}_i, \tau_{ij}\geq 0, D_{ij}<R\}\subset\{1,\ldots,N\}.
\end{equation}
\end{definition}

In the decision-making phase, we define the collision-interaction force being perpendicular to the particle velocity. To this end, we first consider interaction between two particles, then we extend the two-particles model to $N$-particles case. Consider at time $t$, the particle $i$ interacts with the particle  $j\in\mathcal{K}_i^{Co}$. The collision-interaction force for the particle $i$ in this two-particles model is determined by the formulation:
\begin{equation}\label{eq:fcoi}
\cfrac{\rd v_i}{\rd t} = \omega_{ij}^{Co} \|v_i\| e_{\phi_i},
\end{equation}
where
\begin{align}
\omega_{ij}^{Co} &= -C_0 \cos^\varepsilon(\alpha_{ij}) e^{-\tau_{ij}/C_1} \cdot g(\dot\alpha_{ij}),\label{eq:omegaCo}\\
\cos^\varepsilon(\alpha_{ij})&=\cos(\alpha(x_j-x_i,v_i))*\eta_\varepsilon, \label{eq:H_ij}\\
\eta_\varepsilon(y,w)&=\begin{cases}
\frac{1}{\varepsilon^4} C \exp\left({\frac{1}{(\|y\|^2+\|w\|^2)/\varepsilon^2-1}}\right), & \|y\|^2+\|w\|^2<\varepsilon^2, \notag\\
0, & \|y\|^2+\|w\|^2\geq \varepsilon^2,\end{cases}\\
g(\dot\alpha_{ij}) &= \cfrac{2}{1+e^{-\dot\alpha_{ij}/\delta_0}}-1+\delta_1.\label{eq:sign}
\end{align}
 $C_0, C_1, \delta_0, \delta_1$ are all positive parameters and will be specified later. $C$ is chosen properly such that 
\begin{equation*}
\iint_{\RR^2\times\RR^2} \eta_{\varepsilon}(y,w)\rd y\rd w=1.
\end{equation*}
The rotation frequency of collision-avoidance interaction, denoted by $\omega_{ij}^{Co}$, describes how quickly the angular velocity of the particle is and which direction the particle rotates. $\cos^\varepsilon(\alpha_{ij})$ is the regularized function of $\cos\alpha(x,v)$ with the mollification operator $\eta_\varepsilon$. The function $g(\cdot)$ is an approximation of the sign function. Thanks to the function $g(\cdot)$, the particle will turn to the right to avoid the collision when  $\dot\alpha_{ij}>0$, otherwise it turns to the left. Moreover, it is easy to notice that small absolute value of DBA indicating high collision risk. However, the particle fails to change its direction provided that $\dot\alpha_{ij}=0$ yields $g(\dot\alpha_{ij})=0$. Instead we add a small disturbance $\delta_1$ to avoid the failure situation.

\begin{remark}
Compared to~\cite{parzani2017three}, our collision-interaction force is much more smooth.
Moreover, similarly as the Lemma 2.9 and the Lemma 2.12 in \cite{parzani2017three}, we can show that $|\dot\alpha_{ij}|$ will increase exponentially if rules in (\ref{eq:fcoi}) is taken. Thus the collision can be avoided in most cases. 
\end{remark}

Now let us consider the $N$-particles model, it is just an overlying of collision-interaction forces between two particles in the set $\mathcal{K}_i^{Co}(t)$. Therefore, the collision-interaction force for the particle $i$ is written as the mean value of all interactions:
\begin{equation}\label{eq:fco}
F^{Co}_i(t) = \cfrac{1}{\overline{\mathcal{K}_i^{Co}}+\beta}\sum^{N}_{j=1} \omega_{ij}^{Co} \mathbbm{1}_{\mathcal{K}_i^{Co}}(j) \|v_i\| e_{\phi_i},
\end{equation}
where $\overline{\mathcal{K}_i^{Co}}$ denotes the number of elements in the set $\mathcal{K}_i^{Co}$. $\beta\in\RR^+$ is a small constant which makes sure the denominator to be positive, and $\mathbbm{1}_{\mathcal{K}_i^{Co}}$ is the characteristic function.

\begin{remark}
Although each particle needs to interact with all other particles to determine $F^{Co}_i(t)$, we divide the sum of all interactions by $\overline{\mathcal{K}_i^{Co}}$ rather than $N$, since the particle $i$  can only see particles in its ``vision cone'' $\mathcal{C}_i$.
\end{remark}

\subsubsection{The imminent-interaction force}
When the density of particles is too high, especially for particles with high velocity, it is hard to avoid some collisions only by changing the orientation. In these cases, we should decrease  velocity module to avoid collisions.

If the particle $j$ is so close to the particle $i$ and yet has a potential collision risk with the particle $i$, it is required that the particle $i$ should decelerate to avoid the imminent collision.  If MD is smaller than $2R_0$, twice of the radius of particles  (see Remark~\ref{remark:saferadius}), it is noticed that TTI is not the real time to collide for the two particles. It may happen that the distance between the two particles is close to $2R_0$, but TTI is not small enough to reflect the imminent collision between the two particles. Thus, we introduce the time to collide (TTC), meaning the time for the two particles to collide, and its definition is given as follows.
\begin{definition}\label{def:taup}
Consider the particle $i$ and the particle $j$ with $D_{ij}\leq 2R_0$ at time $t$, we define time to collide (TTC), denoted by $\tilde{\tau}_{ij}(t)$, as the time for the two particles $i$ and $j$ to collide (see Figure~\ref{fig:TTC}) under the assumption (A2).
\end{definition}

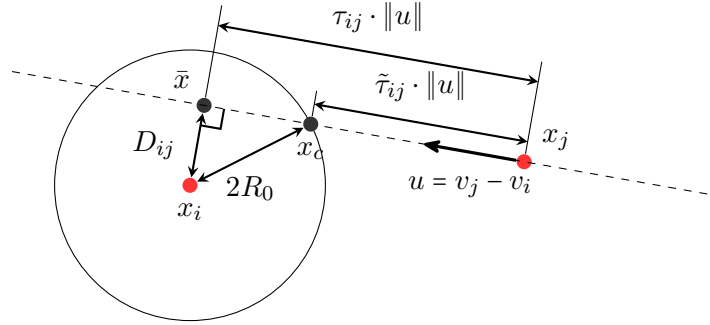
\begin{figure}
\centering
\begin{tikzpicture}[
par/.style={circle,inner sep=0,minimum size=2mm,fill=black!80},
ppar/.style={par,fill=red!80},scale=0.9]
\node [ppar,label=below:$x_i$] (xi) at (0,0) {};
\node [par,label=120:$\bar{x}$] (bxi) at ($(xi)+(80:1.2)$) {};
\node [ppar,label=10:$x_j$] (xj) at ($(bxi)!4!90:(xi)$) {};
\draw [name path=Cir,black] (xi) circle (2);
\draw [<->,>=stealth,thick] (xi) --node [label=left:$D_{ij}$]{} (bxi);
\draw [dashed, name path=Lin] ($(bxi)!-0.6!(xj)$) -- ($(xj)!-0.6!(bxi)$);
\path [name intersections={of=Cir and Lin, name=i}];
\node [par,label=below:$x_c$] (i) at (i-1) {};
\draw [<->,>=stealth,thick] (xi) --node [label=-90:$2R_0$]{} (i);
\draw [thick] ($(bxi)!3mm!(xj)$) -- ($($(bxi)!3mm!(xj)$)!3mm!-90:(xj)$) -- ($(bxi)!3mm!(xi)$);
\draw [->, >= {Stealth[round]},very thick] (xj) --node [label={below:$u=v_j-v_i$}]{} ($(xj)!1.5cm!(bxi)$);
\draw [<->,>=stealth,thick] ($(xj)+(80:3mm)$) --node [label=above:$\tilde{\tau}_{ij}\cdot \|u\|$]{} ($(i)+(80:3mm)$);
\draw [<->,>=stealth,thick] ($(xj)+(80:12mm)$) --node [label=above:${\tau}_{ij}\cdot \|u\|$]{} ($(bxi)+(80:12mm)$);
\draw (i) -- ($(i)+(80:5mm)$)
	(xj) -- ($(xj)+(80:15mm)$)
	(bxi) -- ($(bxi)+(80:15mm)$);
\end{tikzpicture}
\caption{\label{fig:TTC}Illustration of TTI and TTC. $u$ is relative velocity of the particle $j$ to the particle $i$. If particles have no volume, it takes TTC $\tilde{\tau}_{ij}$ for the particle $j$ to $x_c$, and TTI $\tau_{ij}$ to $\bar x$. When the relative position of the particle $j$ is at $\bar x$, the particles $i$ and $j$ are at interaction points $\bar{x}_i, \bar{x}_j$ respectively.}
\end{figure}

The TTC can be practically computed as follows.
\begin{proposition}
Take the same assumption in Definition~\ref{def:taup}, the value of $\tilde{\tau}_{ij}(t_0)$ is given by
\begin{equation}\label{eq:ttau}
\tilde{\tau}_{ij}(t_0) = \tau_{ij} - \cfrac{\sqrt{4R_0^2-D_{ij}^2}}{\|v_i-v_j\|}.
\end{equation}
\end{proposition}

Thanks to $\tilde{\tau}_{ij}$, in the perception phase, we define the imminent-interaction set $\mathcal{K}_i^{Im}(t)$ consisting of particles that have an imminent collision risk with the particle $i$.
\begin{definition}\label{def:R^im}
For a given particle $i$, any particle $j$, satisfying the following four conditions:
\begin{enumerate}
\item $x_j\in \mathcal{C}_i$;
\item  $\tilde\tau_{ij}$ is nonnegative;
\item  $D_{ij}$ is smaller than {a constant threshold} $R^{Im}_0\in (0,R)$;
\item $d_{ij} = \|x_i-x_j\|(t)$ is smaller than a constant threshold $R^{Im}\in(2R_0,\infty)$;
\end{enumerate}
 is in the imminent-interaction set $\mathcal{K}_i^{Im}(t)$. In another word,
\begin{equation}\label{eq:KIm}
\mathcal{K}_i^{Im}(t) = \{j:x_j\in\mathcal{C}_i, \tilde\tau_{ij}\geq 0, D_{ij}<R^{Im}_0, d_{ij}<R^{Im}\}\subset\mathcal{K}_i^{Co}(t).
\end{equation}
\end{definition}

In decision-making phase, consider at time $t$, the particle $i$ interacts with the particle  $j\in\mathcal{K}_i^{Im}$. The imminent-interaction force for the particle $i$ in this two-particles model is determined by the formulation:
\begin{equation}\label{eq:fimi}
\cfrac{\rd v_i}{\rd t} = -\omega_{ij}^{Im} \|v_i\| e_{\rho_i},
\end{equation}
where
\begin{gather}
\omega_{ij}^{Im} = C_2 e^{-d_{ij}\tilde{\tau}_{ij}/C_3},\quad C_2,C_3\in \mathbb{R}^+.\label{eq:omegaim}
\end{gather}
$C_2,C_3$ are positive parameters and will be specified later. The smaller $d_{ij}$ and $\tilde{\tau}_{ij}$ are, the higher collision risk the particle $i$ faces, thus it yields larger $\omega_{ij}^{Im}$. {The acceleration in (\ref{eq:fimi}) is in opposite direction to $v_i$ in order to decelerate the particle $i$ to avoid the imminent collision}. It is noted that the imminent-interaction force only changes  velocity's module, not its orientation.

Now considering $N$-particles model, we averages all interactions for the particles in the imminent-interaction set $\mathcal{K}_i^{Im}$. The imminent-interaction force thus can be described as follows
\begin{equation}\label{eq:fim}
F^{Im}_i(t) = \cfrac{1}{\overline{\mathcal{K}_i^{Im}}+\beta}\sum^{N}_{j=1} \omega_{ij}^{Im} \mathbbm{1}_{\mathcal{K}_i^{Im}}(j) (-v_i),
\end{equation}
where $\overline{\mathcal{K}_i^{Im}}$ denotes the number of elements in the set $\mathcal{K}_i^{Im}$, and $\mathbbm{1}_{\mathcal{K}_i^{Im}}$ is the characteristic function. $\beta\in\RR^+$ takes the same value as in~(\ref{eq:fco}).

\subsubsection{The following-interaction force}
With the collision-interaction force and the imminent-interaction force, one particle modifies its velocity only to avoid collisions with near particles. When the density of particles is too high, some particles may not have enough spacing to avoid collisions and may block up with each other resulting in a dead lock. We hope that particles with the same moving orientation will follow in a line to relieve congestion. Therefore, we introduce the following-interaction force $F^{Fo}_i(t)$ to attain this aim.

So, in the perception phase, we define the following-interaction set $\mathcal{K}_i^{Fo}(t)$ consisting of particles, which have similar orientations as for the particle $i$ and can be followed to relieve congestion.
\begin{definition}\label{def:R^fo}
For a given particle $i$, any particle $j$  satisfying the following three conditions:
\begin{enumerate}
\item $x_j\in \mathcal{C}_i$; 
\item $\tau_{ij}$ is negative;
\item The distance between each other $d_{ij} = \|x_i-x_j\|(t)$ is smaller than a constant threshold $R^{Fo} \in(R_0,\infty)$;
\end{enumerate}
is in following-interaction set $\mathcal{K}_i^{Fo}(t)$. In another word,
\begin{equation}\label{eq:KFo}
\mathcal{K}_i^{Fo}(t) = \{j: x_j\in\mathcal{C}_i, \tau_{ij}<0, d_{ij}< R^{Fo}\}.
\end{equation}
\end{definition}

In the decision-making phase, consider at time $t$, the particle $i$ interacts with the particle $j\in\mathcal{K}_i^{Fo}$. The following-interaction force for the particle $i$ in this two-particles model is determined by the formulation:
\begin{equation*}
\cfrac{\rd v_i}{\rd t} = \omega_{ij}^{Fo} \|v_i\| e_{\phi_i},
\end{equation*}
where
\begin{align}
\omega_{ij}^{Fo} &=  C_4 e^{-|\dot\alpha_{ij}|d_{ij}^2/C_5}\sin^\varepsilon(2\alpha_{ij}), \quad C_4,C_5\in \mathbb{R}^+,\label{eq:omegafo}\\
\sin^\varepsilon(2\alpha_{ij})&=\sin\left(2\alpha(x_j-x_i,v_i)\right)*\eta_\varepsilon.\label{eq:sinus}
\end{align}
$C_4,C_5$ are positive parameters and will be specified later. Intuitively, the particle $i$ shall follow the particle $j$, which is in front of the particle $i$ and moves in the similar orientation. In other words, it is required that their distance $d_{ij}$ is proper ,and $|\dot\alpha_{ij}|$ and  $|\alpha_{ij}|$ are both small. It is interesting to notice that in the following-interaction force, small $|\dot\alpha_{ij}|$ does not indicate high collision risk while it is the case for the collision-avoidance force. The term $e^{-|\dot\alpha_{ij}|d_{ij}^2/C_5}$ in~(\ref{eq:omegafo}) leads the particle $i$ to follow an appropriate leader with low $|\alpha_{ij}|$ and close distance $d_{ij}$, while the term $\sin^\varepsilon(2\alpha_{ij})$ is to decrease the absolute value of the bearing angle $\alpha_{ij}$ to follow up the particle $j$.

In $N$-particles model, the following-interaction force for the particle $i$ satisfies the following formulation:
\begin{equation}\label{eq:ffo}
F^{Fo}_i(t) = \cfrac{1}{\overline{\mathcal{K}_i^{Fo}}+\beta}\sum^{N}_{j=1} \omega_{ij}^{Fo} \mathbbm{1}_{\mathcal{K}_i^{Fo}(j)} \|v_i\|e_{\phi_i}.
\end{equation}

\subsubsection{Influence of obstacles and destinations}
The influence of obstacles and targets is remained to be discussed. Obstacles always have irregular shape and can be seen as particular particles. We define the equivalent point for both obstacles and destinations.
\begin{definition}\label{equivalentpoint}
 Consider a rigid body, denoted by a simply connected subset $O(t)\subset\RR^2$, which has velocity and  compact volume, then  we introduce the equivalent point $x_O$ and the equivalent velocity $v_O$ of $O(t)$ with respect to the particle $i\in\{1,\ldots,N\}$ as follows.
\begin{itemize}
\item The equivalent point $x_O\in O(t)$ is the closest point of $O(t)$ related to the particle $i$, {\it i.e.}
\begin{equation*}
x_O = \underset{x\in \partial O(t) \cap \mathcal{K}^{Co}_i(t)}{\arg\, \min} d(x_i(t),  x),
\end{equation*}
\item The equivalent velocity $v_O$ of $O(t)$ is defined to be the velocity corresponding to  the equivalent point $x_O$.
\end{itemize}
\end{definition}

For any particle $i$, the obstacle $O(t)\subset\RR^2$ is regarded as a particle with position $x_O$ and velocity $v_O$ in perception phase. Therefore, the value of $N$ in~(\ref{eq:fco}) indeed is the total number of particles and obstacles.

In the decision-making phase, in order to lead each particle to its destinations, the exit-interaction force $F^{Ex}_i$ is needed. Let us consider the particle $i$ with the destination set $T\subset\RR^2$ at time $t$. Applying  the same strategy in obstacles, we can define the equivalent point $x_T$ for $T$. Therefore, the exit-interaction force $F^{Ex}_i$ can be given as follows
\begin{equation}\label{eq:fex}
F^{Ex}_i = -\nabla V(x_i) - \sigma v_i, 
\end{equation}
where $\sigma>0$ represents a friction to restrict the maximum velocity, and $V(\cdot)$ is a potential function reaching its minimum at the destination. 

In summary, the agent-based collision-avoidance model consists of perception phase and decision-making phase. In perception phase, each particle observes from its ``vision cone''. For any particle $i\in\{1,\ldots,N\}$, all particles in its ``vision cone'' $\mathcal{C}_i$ will be grouped into four parts: the collision-interaction set $\mathcal{K}_i^{Co}(t)$, the imminent-interaction set $\mathcal{K}_i^{Im}(t)$, the following-interaction set $\mathcal{K}_i^{Fo}(t)$ and others left.  In decision-making phase, particles choose their local optimal motion from the analysis in the perception phase. We suppose all the parameters are regarded as scalars in equations and the model is measured by meter in space and second in time. It is assumed that each particle or one obstacle has a unit mass, and the acceleration of each particle is equal to the value of force with $1m/s^2$ unit. From \eqref{eq:fco},\eqref{eq:fim},\eqref{eq:ffo},\eqref{eq:fex}, the final acceleration for the particle $i\in\{1,\ldots,N\}$ depends on the value of the total force $F_i(t)$,
\begin{equation*}
F_i(t) = F^{Co}_i(t)+F^{Im}_i(t) + F^{Fo}_i(t) + F^{Ex}_i(t).
\end{equation*}

In conclusion, each particle obeys the rules as follows:
\begin{equation}\label{eq:modelODE}
\begin{cases}
\displaystyle\cfrac{\rd x_i}{\rd t} = v_i, \\[3mm]
\displaystyle\cfrac{\rd v_i}{\rd t} = F_i(t) := F^{Co}_i(t)+F^{Im}_i(t) + F^{Fo}_i(t)+F^{Ex}_i(t),
\end{cases}
i\in\{1,\ldots,N\}.
\end{equation}
\begin{remark}
It's noted that in some extreme situations, the collision is unavoidable for our model, such as the too fast velocity and too short distance for particles. Especially when the particles are too many in a restricted region, the congestion is unavoidable.
\end{remark}
\subsection{Mean-field limit}\label{sec:Meanfield}
In this part, we extend the $N$-particles model~(\ref{eq:modelODE}) to its corresponding mean-field limit. Let us first introduce the so-called empirical distribution as in~\cite{parzani2017three} or~\cite{spohn2012large}, denoted by $f^N(t,x,v)$,  which is defined as follow:
\begin{equation*}
f^N(t,x,v) := \cfrac{1}{N} \sum_{i=1}^{N} \delta(x-x_i)\delta(v-v_i),   
\end{equation*}
where $(x_i,v_i)_{i=1}^N$ is the solution to the $N$-particles system~(\ref{eq:modelODE}) and {$\delta(\cdot)$ is a Dirac function.}

Then  as we did in the perception phase of the agent-based modelling, we introduce the ``vision cone'' $\mathcal{C}(v)$ (centered at the origin with angel $\arccos\kappa$ about the direction $v\in\RR^2$), the minimal distance $D(z,u)$ and the time to interaction $\tau(z,u)$ for continuous variables as follows%
\begin{align*}
\mathcal{C}(v) &:=\{z\in\RR^2|<z,v>\geq\kappa\|z\|\cdot\|v\|\}\subset\RR^2, \\
\displaystyle D(z,u) &:=(\|z\|^2-(\frac{<z,u>}{\|u\|})^2)^{1\over2}, \quad \forall u\in\RR^2,\\[3mm]
\displaystyle\tau(z,u)&:=-\frac{<z,u>}{\|u\|^2}, \quad \forall u\in\RR^2.\\[3mm]
\end{align*}
Thanks to the above definition, we define the sets $I^{Co}(u),I^{Im}(u),I^{Fo}(u)\subset\RR^2$ for the collision-interaction set, the imminent-interaction set and the following-interaction set respectively
\begin{align}
I^{Co}(u) &:=\{z\in\RR^2|\tau(z,u)\geq 0,D(z,u)<R\}, \quad\forall u\in\RR^2,\notag\\
I^{Im}(u)&:=\{z\in\RR^2|\tau(z,u)\geq0,D(z,u)<R^{Im}_0,\|z\|<R^{Im}\}, \quad\forall u\in\RR^2, \notag\\
I^{Fo}(u)&:= \{z\in\RR^2 | \tau(z,u)<0, \|z\|<R^{Fo}\}, \quad\forall u\in\RR^2, \notag
\end{align}
where the definition of $R$ and $R_0$ can be seen in Remark~\ref{remark:saferadius}, the definition of $R^{Im}, R^{Fo}$ can be seen in Definition~\ref{def:R^im} and~\ref{def:R^fo}. 
Finally, by taking intersection with the vision cone $\mathcal{C}(v)$, we obtain the sets $\mathcal{K}^{Co}(v,w),\mathcal{K}^{Im}(v,w), \mathcal{K}^{Fo}(v,w)$ as follows
\begin{gather}
\mathcal{K}^{Co}(v,w):=I^{Co}(w-v)\cap \mathcal{C}(v), \quad \forall v,w\in\RR^2,\notag\\
\mathcal{K}^{Im}(v,w) :=I^{Im}(w-v)\cap \mathcal{C}(v), \quad \forall v,w\in\RR^2,\notag\\
\mathcal{K}^{Fo}(v,w) := I^{Fo}(w-v)\cap \mathcal{C}(v), \quad \forall v,w\in\RR^2.\notag
\end{gather}

Now for the decision-making phase, thanks to the definition of the empirical distribution, we can write the collision-interaction force for continuous variable as
\begin{align*}
\mathcal{F}^{Co}(f^N) = \Omega_{Co}^N\cdot v^\bot,
\end{align*}
where we set $v^\bot$ as $v$ rotated counterclockwise with ${\pi/2}$ and
\begin{align}
\displaystyle\Omega_{Co}^N(t,x,v) &= \cfrac{1}{\lambda^{Co}(t,x,v)} \iint\limits_{\RR^2\times\RR^2} m^{Co}(y-x,v,w)\mathbbm{1}_{\mathcal{K}^{Co}(v,w)}(y-x)f^N(t,y,w)\rd y\rd w, \notag\\
\displaystyle\lambda^{Co}(t,x,v) &= \iint\limits_{\RR^2\times\RR^2} \mathbbm{1}_{\mathcal{K}^{Co}(v,w)}(y-x)f^N(t,y,w)\rd y\rd w+\beta, \notag\\
m^{Co}(z,v,w) &= -C_0 \cos^\varepsilon(\alpha(z,v))e^{-\tau(z,w-v)/C_1}\cdot g(\cfrac{(w-v)\times z}{\|z\|^2}), \notag\\
\cos^\varepsilon(\alpha(z,v)) &=\cos(\alpha(z,v))*\eta_\varepsilon,\notag\\
g(x) &= \cfrac{2}{1+e^{-x/\delta_0}}-1+\delta_1.\notag
\end{align}

Similarly, the imminent-interaction force is given  by
\begin{align*}
\mathcal{F}^{Im}(f^N) = -\Omega_{Im}^N\cdot v,
\end{align*}
with
\begin{align}
\displaystyle\Omega_{Im}^N(t,x,v) &= \cfrac{1}{\lambda^{Im}(t,x,v)}\iint\limits_{\RR^2\times\RR^2} m^{Im}(y-x,v,w)\mathbbm{1}_{\mathcal{K}^{Im}(v,w)}(y-x)f^N(t,y,w)\rd y\rd w, \notag\\
\displaystyle\lambda^{Im}(t,x,v) &= \iint\limits_{\RR^2\times\RR^2} \mathbbm{1}_{\mathcal{K}^{Im}(v,w)}(y-x)f^N(t,y,w)\rd y\rd w+\beta, \notag \\
m^{Im}(z,v,w) &= C_2 e^{-\|z\|\tilde\tau(z,w-v)/C_3}, \notag\\
\displaystyle\tilde\tau(z,u) &:= \tau(z,u) - \cfrac{\sqrt{\max(4R_0^2 - D(z,u)^2,0)}}{\|u\|^2}, \quad \forall u\in\RR^2. \notag
\end{align}
The following-interaction force is given by
\begin{align*}
\mathcal{F}^{Fo}(f^N) = \Omega_{Fo}^N\cdot v^\bot,
\end{align*}
with
\begin{align}
\displaystyle\Omega_{Fo}^N(t,x,v) &= \cfrac{1}{\lambda^{Fo}(t,x,v)}\iint\limits_{\RR^2\times\RR^2} m^{Fo}(y-x,v,w)\mathbbm{1}_{\mathcal{K}^{Fo}(v,w)}(y-x)f^N(t,y,w)\rd y\rd w, \notag\\
\displaystyle\lambda^{Fo}(t,x,v) &= \iint\limits_{\RR^2\times\RR^2} \mathbbm{1}_{\mathcal{K}^{Fo}}(v,w)(y-x)f^N(t,y,w)\rd y\rd w+\beta, \notag \\
m^{Fo}(z,v,w) &= C_4 e^{-\|\dot\alpha(z,w-v)\|\|z\|^2/C_5}\sin^\varepsilon(2\alpha(z,v)), \notag\\
\sin^\varepsilon(2\alpha(z,v)) &= \sin(2\alpha(z,v)) * \eta_\varepsilon,\notag\\
\dot\alpha(z,u) &= \cfrac{z\times u}{\|z\|^2}.\notag
\end{align}

Therefore, according to the conservation laws, the empirical distribution $f^N(t,x,v)$ satisfies the kinetic equation as follow:
\begin{equation*}
\partial_t f^N + v\cdot \nabla_x f^N - \nabla_x V\cdot \nabla_v f^N + \nabla_v \cdot ((\Omega_{Co}^N\cdot v^\bot-\Omega_{Im}^N\cdot v+\Omega_{Fo}^N\cdot v^\bot) f^N) = \sigma \nabla_v\cdot(vf^N).
\end{equation*}

Now by setting the volume radius of one particle $R_0=0$ and passing $N\to\infty$, we obtain the mean-field limit
\begin{equation}\label{eq:meanfield}
\displaystyle \partial_t f + v\cdot \nabla_x f + \nabla_v(\mathcal{F}(f)f)=0,
\end{equation}
where the force field is given by
\begin{equation*}
\mathcal{F}(f) = - \nabla_x V + (\Omega_{Co}\cdot v^\bot-\Omega_{Im}\cdot v+\Omega_{Fo}\cdot v^\bot) - \sigma v,
\end{equation*}
with
\begin{equation*}
\begin{cases}
\displaystyle\Omega_{Co}(t,x,v)=\cfrac{1}{\lambda^{Co}(t,x,v)} \iint\limits_{\RR^2\times\RR^2} m^{Co}(y-x,v,w)\mathbbm{1}_{\mathcal{K}^{Co}(v,w)}(z)f(t,y,w)\rd y\rd w,\\[3mm]
\displaystyle\Omega_{Im}(t,x,v) = \cfrac{1}{\lambda^{Im}(t,x,v)}\iint\limits_{\RR^2\times\RR^2} m^{Im}(y-x,v,w)\mathbbm{1}_{\mathcal{K}^{Im}(v,w)}(z)f(t,y,w)\rd y\rd w, \\[3mm]
\displaystyle\Omega_{Fo}(t,x,v)=\cfrac{1}{\lambda^{Fo}(t,x,v)} \iint\limits_{\RR^2\times\RR^2} m^{Fo}(y-x,v,w)\mathbbm{1}_{\mathcal{K}^{Fo}(v,w)}(z)f(t,y,w)\rd y\rd w,\\[3mm]
\displaystyle\lambda^{Co}(t,x,v) = \iint\limits_{\RR^2\times\RR^2} \mathbbm{1}_{\mathcal{K}^{Co}(v,w)}(y-x)f(t,y,w)\rd z\rd w,\\[3mm]
\displaystyle\lambda^{Im}(t,x,v) = \iint\limits_{\RR^2\times\RR^2} \mathbbm{1}_{\mathcal{K}^{Im}(v,w)}(y-x)f(t,y,v_j)\rd y\rd w, \\[3mm]
\displaystyle\lambda^{Fo}(t,x,v) = \iint\limits_{\RR^2\times\RR^2} \mathbbm{1}_{\mathcal{K}^{Fo}(v,w)}(y-x)f(t,y,w)\rd y\rd w,\\[3mm]
f(t=0)=f_0\in L^1\cap L^\infty (\RR^4),\\[3mm]
\iint\limits_{\RR^2\times\RR^2} f_0(x,v)\rd x\rd v = 1,
\end{cases}
\end{equation*}

and $V(x)$ is a regular potential function. For the equation~(\ref{eq:meanfield}), we can prove the existence of the weak solution in sense of \cite{carrillo2018mean} analogous to Theorem~3.1 stated in \cite{parzani2017three}.
\begin{theorem}\label{th:energy}
Consider a smooth potential $V(x)\geq0$ and $V\in C^1(\RR^2)$. Assume that $f_0\in L^1\cap~L^\infty(\RR^2\times\RR^2)$, with $f_0\geq0$ and
\begin{equation*}
\iint\limits_{\RR^2\times\RR^2} (|x|^2+|v|^2)f_0(x,v)\rd x\rd v<\infty.
\end{equation*}
Then for any $T>0$, there exists a weak solution to the equation~(\ref{eq:meanfield}) in sense of \cite{carrillo2018mean} such that
\begin{equation*}
f(t)\in L^1\cap L^\infty(\RR^2\times\RR^2),\quad \iint\limits_{\RR^2\times\RR^2} (|x|^2+|v|^2)f(t,x,v)\rd x\rd v<C(T,f_0)\quad a.e. \ t\in[0,T],
\end{equation*}
and for any $\varphi\in C_c^\infty([0,T)\times\RR^2\times\RR^2)$,
\begin{equation}\label{eq:weaksolution}
\begin{split}
\int_0^T\iint\limits_{\RR^2\times\RR^2} f(t)(\partial_t\varphi+v\cdot\nabla_x\varphi-(\nabla_xV-\Omega_{Co}\cdot v^\bot +\Omega_{Im}\cdot v-\Omega_{Fo}\cdot v^\bot\\
+\sigma v)\cdot \nabla_v\varphi)\rd x\rd v\rd t+\iint\limits_{\RR^2\times\RR^2}f_0\varphi(0)\rd x\rd v=0. 
\end{split}
\end{equation}
In addition, the system's energy satisfies the inequality below,
\begin{equation*}
{\rd\ \over\rd t}\iint\limits_{\RR^2\times\RR^2}\left(\cfrac{|v|^2}{2}+V(x)\right)f(t,x,v)\rd x\rd v\leq -\iint\limits_{\RR^2\times\RR^2}(\sigma+\Omega_{Im})|v|^2f(t,x,v)\rd x\rd v.
\end{equation*}
\end{theorem}

Moreover, the flow generated from~\eqref{eq:meanfield} is well defined thanks to the regularity of the force field $F(f)$, stated in Theorem~\ref{th:lip}.

\begin{theorem}\label{th:lip}
Consider $\nabla_x V(x)$ to be Lipschitz continuous. Let $f$ be a weak solution to (\ref{eq:weaksolution}) given by Theorem~\ref{th:energy}. Assume that $f_0\in L^1\cap~L^\infty(\RR^2\times\RR^2)$, with $f_0\geq0$ and $f_0$ has compactly support in phase space.
Then the force field $F(f)$ generated by $f$ is locally Lipschitz continuous in phase space uniformly on $[0, T]$. More precisely, there exists a constant $C>0$ depending on $\|f\|_{L^1\cap L^\infty}$ and the support of $f$ in phase space, such that
\begin{equation*}
|\mathcal{F}(f)(t,x,v)-\mathcal{F}(f)(t,\tilde x,\tilde v)|\leq C(1+\|v\|)\left\|\begin{pmatrix}x\\v\end{pmatrix}-\begin{pmatrix}\tilde x\\ \tilde v\end{pmatrix}\right\|,
\end{equation*}
for all $x,v,\tilde x,\tilde v\in\RR^2$ and $t\in[0,T]$.
\end{theorem}
The proof of this theorem is reported in the appendix.

\section{An efficient algorithm}\label{sec:Algorithm}
According to Section 2, one particle only needs to interact with particles in its ``vision cone''. However, to determine particles in the "vision cone", we have to compute all $N(N-1)/2$ relative relationship between particles, therefore the computational complexity is of order $\mathcal{O}(N^2)$. The computational cost is huge when the number of particles is large. In this section, we propose an efficient algorithm to solve the model in Section~\ref{sec:Modeling} by taking inspiration from the Random Batch Method (RBM) in \cite{jin2020random}. Moreover, considering the collision-avoidance mechanism, we propose a hybrid resolution strategy by combining the so-called Cell-List approach \cite{jin2022mean} with RBM.

\subsection{The RBM model for collision avoidance}
Since the computational complexity of (\ref{eq:modelODE}) is of order $\mathcal{O}(N^2)$, it costs too much to solve these ODEs when the number of particles $N$ is very large. To reduce the computational complexity, one may consider a mean-field approach, where we transfer the model to a Fokker-Planck equation (see equation~(\ref{eq:meanfield})) from a mesoscopic view~\cite{lasry2007mean}. The influence on the dynamics of a single particle is replaced by an averaged one. But it is still hard to solve the mean-field limit model with fine resolution. Therefore, it is meaningful to construct a simple method to solve the model efficiently. RBM uses the idea of subsampling when we computes interactions between all particles randomly to approximate the exact solution, where the computational complexity of the model can be reduced to $\mathcal{O}(N)$. Our aim is to use the idea of RBM to solve the collision-avoidance model and introduce its corresponding mean-field limit.

The acceleration of the particle $i$ consists of four parts: the collision-interaction force $F^{Co}_i$, the imminent-interaction force $F^{Im}_i$, the following-interaction force $F^{Fo}_i$ and the exit-interaction force $F^{Ex}_i$. From (\ref{eq:modelODE}), we know that the computational complexity of the first three forces $F^{Co}_i,F^{Im}_i,F^{Fo}_i$ for each particle $i$ is of order $\mathcal{O}(N^2)$, since the particle $i$ needs to interact with all other $N-1$ particles. Instead, the idea of RBM is to divide all particles into many small batches (only $p$ particles in each batch and $p$ is prefixed, small and independent of $N$), and only to compute interactions of particles within the same batch, therefore the computational complexity of these three forces is reduced to $\mathcal{O}(pN)$.

Now, let us describe the algorithm with RBM. First, we introduce a partition of time interval $[0,T]$ as $\{t_0,t_1,\cdots,t_m\}$.  In each subinterval $[t_{k-1},t_{k})$, the $N$ particles $((n-1)p<N\leq np)$ are divided into $n$ small batches with size at most $p$ ($p\ll N$, often $p=2$) randomly, denoted by $C_q,q=1,\ldots,n$. In perception phase, each particle only interacts within the batch where it belongs during $t\in[t_{k-1},t_{k})$.

For any particle $i\in\{1,\ldots,N\}$, there exists  $q\in\{1,\ldots,n\}$, such that $i$ belongs to the batch $C_{q}$. In perception phase, the three interaction sets in RBM can be determined as follows:

\begin{align*}
\mathcal{K}^{Co,RBM}_i &:= \left\{j\neq i:j\in\mathcal{K}^{Co}_i(t),j\in C_{q}\right\} = \mathcal{K}^{Co}_i(t)\cap C_{q},\\
\mathcal{K}^{Im,RBM}_i &:=  \mathcal{K}^{Im}_i(t)\cap C_{q},\\
\mathcal{K}^{Fo,RBM}_i &:= \mathcal{K}^{Fo}_i(t)\cap C_{q},
\end{align*}
where the definitions of $\mathcal{K}^{Co}_i,\mathcal{K}^{Im}_i,\mathcal{K}^{Fo}_i$ can be referred in~\eqref{eq:KCo},\eqref{eq:KIm},\eqref{eq:KFo} respectively.

Similarly, in decision-making phase, the corresponding three interaction forces can be determined as follows:
\begin{align*}
F^{Co,RBM}_i(t) &= \cfrac{1}{\overline{\mathcal{K}^{Co,RBM}_i}+\beta}\sum_{j\in C_{q},j\neq i} \omega^{Co}_{ij} \mathbbm{1}_{\mathcal{K}^{Co,RBM}_i}(j) v^\bot_i, \\
F^{Im,RBM}_i(t) &=-\cfrac{1}{\overline{\mathcal{K}^{Im,RBM}_i}+\beta}\sum_{j\in C_{q},j\neq i} \omega^{Im}_{ij} \mathbbm{1}_{\mathcal{K}^{Im,RBM}_i}(j) v_i, \\
F^{Fo,RBM}_i(t) &=\cfrac{1}{\overline{\mathcal{K}^{Fo,RBM}_i}+\beta}\sum_{j\in C_{q},j\neq i} \omega^{Fo}_{ij} \mathbbm{1}_{\mathcal{K}^{Fo,RBM}_i}(j) v^\bot_i.
\end{align*}


Thus, in each time step $[t_{k-1},t_{k})$, it contains two steps:
\begin{enumerate}
\item[(i)] Divide particles into $n$ batches randomly;
\item[(ii)] Evolve interactions only within batches. 
\end{enumerate}
The RBM model for collision avoidance obeys the rules as follows:
\begin{equation}\label{eq:RBM}
\begin{cases}
\cfrac{\rd x_i}{\rd t} = v_i, \\[3mm]
\cfrac{\rd v_i}{\rd t} = F^{RBM}_i(t):=F^{Co,RBM}_i(t)+F^{Im,RBM}_i(t)+F^{Fo,RBM}_i(t)+F^{Ex}_i(t),
\end{cases}\quad i\in\{1,\ldots,N\}.
\end{equation}

Clearly, with RBM each particle only needs to interact with $p-1$ particles in each time step, therefore the computational complexity is reduced from $\mathcal{O}(N^2)$ to $\mathcal{O}(pN)$. The details are described in Algorithm~\ref{ta:RBM}.


\begin{algorithm}
\caption{\label{ta:RBM}the RBM model for collision avoidance}
\begin{algorithmic}[1]
\State Given a time partition $t_k,k\in\{0,1,\cdots,m\}$.
\For{$k$ in $1:m$}
\State Divide $\{1,2,\ldots,N\}$ into $n=N/p$ batches randomly.
\For{each batch $C_q$}
\State Update $(x_i,v_i)$($i\in C_q$) by solving the following ODEs with $t\in[t_{k-1},t_k)$,
\Statex
\thetag{\ref{eq:RBM}}$\hfill\begin{cases}
\displaystyle\cfrac{\rd x_i}{\rd t} = v_i, \\[3mm]
\displaystyle\cfrac{\rd v_i}{\rd t} = F^{RBM}_i(t).
\end{cases}\hfill\hfill$
\EndFor
\EndFor
\end{algorithmic}
\end{algorithm}

\subsection{Mean-field limit of the RBM model}
As pointed out in \cite{jin2022mean}, the RBM model can also be viewed as a new model for the collision-avoidance model, in which particles interact  randomly with particles in selected batch, rather than a numerical method. In this part, we investigate the mean-field limit of the RBM model. In a subinterval $[t_{k-1},t_k)$,  the probability of $p$ chosen particles in a batch being correlated will converge to $0$ as  $N\to\infty$. If we assume that all particles' distribution are identical and independent at $t=0$, the marginal distributions of $p$ chosen particles in the batch will also be identical because the particles are exchangeable. Hence, we focus on one fixed particle, $i=1$, to construct the mean-field limit of the collision-avoidance model with RBM.

Let the initial distribution of all particles be independent and identically distributed, {\it i.e.}
\begin{equation*}
(x_i(0),v_i(0))\sim \mu_0, \qquad i\in\{1,\ldots,N\},
\end{equation*}
where $\mu_0$ is an initial distribution function. Notice that when $N\to\infty$, randomly picked $p-1$ particles are identical and independent from the particle $1$ in each time step. The particles in this batch, including the particle $1$, follow the same distribution, and interact with each other until the end of this time step. In the next time step, we draw another $p-1$ particles to the particle $1$ to build a new batch. As a result, when $N\to\infty$, the mean-field limit of $N$-particles system is reduced to a $p$-particles system shown in the Algorithm~\ref{ta:mRBM}.

\begin{algorithm}
\caption{\label{ta:mRBM}Mean-field limit for the RBM model}
\begin{algorithmic}[1]
\State Given an initial distribution function $\tilde{\mu}(x,v,0)=\mu_0$ and a time partition $t_k,k=\{0,\cdots,m\}$.
\For{$k$ in $1:m$}
\State Let $\rho^{(p)}(\cdots,t_{k-1})=\tilde{\mu}(\cdot,\cdot,t_{k-1})^{\otimes p}$, a probability measure on $(\RR^2\times\RR^2)^{p}$, be initial data.
\State Evolve $\rho^{(p)}$ by solving the following Fokker-Planck equation with $t\in[t_{k-1},t_k)$.
\Statex \refstepcounter{equation}\thetag{\theequation}\label{eq:RBMmeanfield}$\hfill
\displaystyle\partial_t \rho^{(p)}+\sum_{i=1}^{p}v_i\cdot \nabla_{x_i}\rho^{(p)}-\sum_{i=1}^{p}\nabla_{x_i}V(x_i) \cdot\nabla_{v_i} \rho^{(p)}\hfill\hfill\hfill$
\Statex $\hfill\displaystyle+\sum_{i=1}^{p}\nabla_{v_i}\cdot\left((\Omega^{Co}_i\cdot v^\bot_i-\Omega^{Im}_i\cdot v_i+\Omega^{Fo}_i\cdot v^\bot_i)\rho^{(p)}\right) =\sigma \sum_{i=1}^{p} \nabla_{v_i}\cdot(v_i\rho^{(p)}).$
\State Set
\Statex $\hfill\displaystyle
\tilde{\mu}(x,v,t_{k}):=\iint\limits_{\RR^{2}\times\RR^2}\cdots\iint\limits_{\RR^{2}\times\RR^2} \rho^{(p)}(x,x_2,\cdots,x_p,v,v_2,\cdots,v_p,t^-_{k})\rd x_2\rd v_2\ldots\rd x_p\rd v_p.\hfill$
\EndFor
\end{algorithmic}
\end{algorithm}
The magnitude of the forces in the equation~\eqref{eq:RBMmeanfield} are defined as follows:
\begin{gather}
\Omega^{Co}_i(\{(x_j,v_j)\}_{j=1}^{p}) := \cfrac{1}{\overline{\mathcal{K}^{Co}_i}+\beta}\sum^{p}_{j=1,j\neq i} \omega^{Co}_{ij} \mathbbm{1}_{\mathcal{K}^{Co}_i}(j)\notag,\\
\Omega^{Im}_i(\{(x_j,v_j)\}_{j=1}^{p}) := -\cfrac{1}{\overline{\mathcal{K}^{Im}_i}+\beta}\sum^{p}_{j=1,j\neq i} \omega^{Im}_{ij} \mathbbm{1}_{\mathcal{K}^{Im}_i}(j)\notag,\\
\Omega^{Fo}_i(\{(x_j,v_j)\}_{j=1}^{p}) := \cfrac{1}{\overline{\mathcal{K}^{Fo}_i}+\beta}\sum^{p}_{j=1,j\neq i} \omega^{Fo}_{ij}\mathbbm{1}_{\mathcal{K}^{Fo}_i}(j)\notag.
\end{gather}

Different from previous section, in this part $\{(x_j,v_j)\}_{j=1}^{p}$ represent $p$ continuous variables in $\RR^2\times\RR^2$. The sets $\mathcal{K}^{Co}_i$, $\mathcal{K}^{Im}_i$, $\mathcal{K}^{Fo}_i$ and the functions $\omega^{Co}_{ij}$, $\omega^{Im}_{ij}$, $\omega^{Fo}_{ij}$ take the same definitions as in Section~\ref{sec:Modeling}, but localized only in a batch.

For the mean field model of the RBM model (\ref{eq:RBMmeanfield}), an analogous consequence of the weak solution as Theorem~\ref{th:energy} can be proved.
\begin{theorem}
Consider a smooth positive potential  $V(x)\in C^1(\RR^2)$. Assume that $\mu_0\in L^1\cap L^\infty(\RR^{2}\times\RR^2)$ with $\mu_0 \geq 0$ and
\begin{equation*}
\iint\limits_{\RR^{2}\times\RR^2}(|x|^2+|v|^2) \mu_0(x,v)\rd x\rd v < \infty.
\end{equation*}
Then for any $T>0$ and almost every time partition $t_k,k=\{0,\cdots,m\}$ in $[0,T]$, there exists a weak solution to equation~({\rm\ref{eq:RBMmeanfield}}) in sense of \cite{carrillo2018mean}, that means for any $\varphi\in C_c^\infty([0,T)\times\RR^2\times\RR^2)$,
\begin{multline*}
\begin{split}
\int_0^T\iint\limits_{\ \RR^2\times\RR^2}\cdots \iint\limits_{\ \RR^2\times\RR^2}\rho^{(p)}(\partial_t\varphi+\sum_{i=1}^{p}v_i\cdot \nabla_{x_i}\varphi\\
-\sum_{i=1}^{p}((\nabla_{x_i}V(x_i)-\Omega^{Co}_i\cdot v^\bot_i &+\Omega^{Im}_i\cdot v_i-\Omega^{Fo}_i \cdot v^\bot_i+\sigma v)\cdot \nabla_{v_i}\varphi)\rd x_1\rd v_1\ldots\rd x_p\rd v_p\rd t
\end{split}\\
+\iint\limits_{\ \RR^2\times\RR^2}\cdots \iint\limits_{\ \RR^2\times\RR^2}\rho^{(p)}(0)\varphi(0)\rd x_1\rd v_1\ldots\rd x_p\rd v_p=0.
\end{multline*}
And we have $\rho^{(p)}(t)\in L^1\cap L^\infty((\RR^2\times\RR^2)^p),$
\begin{equation*}
\iint\limits_{\RR^{2}\times\RR^2}\cdots\iint\limits_{\RR^{2}\times\RR^2} \sum_{i=1}^p(|x_i|^2+|v_i|^2)\rho^{(p)}(\cdots,t)\rd x_1\rd v_1\ldots\rd x_p\rd v_p<C(T,\mu_0)\quad a.e.\ t\in[0,T].
\end{equation*}
In addition, the system's energy satisfies
\begin{equation}\label{eq:RBMenergy}
\begin{split}
{\rd\over\rd t}\iint\limits_{\RR^2\times\RR^2}\left(\cfrac{|v|^2}{2}+V(x)\right)\tilde{\mu}(x,v,t)\rd x\rd v\leq -\sigma\iint\limits_{\RR^2\times\RR^2}|v|^2\tilde{\mu}(x,v,t)\rd x\rd v\\
-\iint\limits_{\ \RR^2\times\RR^2}\cdots \iint\limits_{\ \RR^2\times\RR^2} \Omega^{Im}_1 |v_1|^2\rho^{(p)}(\cdots,t)\rd x_1\rd v_1\ldots\rd x_p\rd v_p.
\end{split}
\end{equation}
Particularly, when $C_2=0$ in \eqref{eq:omegaim}, the inequality~(\ref{eq:RBMenergy}) can be written as
\begin{equation*}
{\rd\over\rd t}\iint\limits_{\RR^2\times\RR^2}\left(\cfrac{|v|^2}{2}+V(x)\right)\tilde{\mu}(x,v,t)\rd x\rd v\leq -\sigma\iint\limits_{\RR^2\times\RR^2}|v|^2\tilde{\mu}(x,v,t)\rd x\rd v.
\end{equation*}
\end{theorem}

\begin{proof}
We will prove the existence of the weak solution by  induction. For this, let us first focus on subinterval $[0,t_1)$. Clearly, for the initial data, we have
\begin{equation*}
\begin{split}
&\iint\limits_{\ \RR^2\times\RR^2}\cdots \iint\limits_{\ \RR^2\times\RR^2} \sum_{i=1}^p(|x_i|^2+|v_i|^2)\rho^{(p)}(\cdots,0)\rd x_1\rd v_1\ldots\rd x_p\rd v_p \\
=& \sum_{i=1}^p \iint\limits_{\RR^{2}\times\RR^2} (|x_i|^2+|v_i|^2) \mu_0(x_i,v_i)\rd x_i \rd v_i
=p\iint\limits_{\RR^{2}\times\RR^2} (|x|^2+|v|^2) \mu_0(x,v)\rd x \rd v<\infty.
\end{split}
\end{equation*}
Thanks to Theorem~\ref{th:energy}, there exits a weak solution $\rho^{(p)}$ for (\ref{eq:RBMmeanfield}) such that the second moment of $\rho^{(p)}(\dots,t)$ is bounded for almost every $t\in[0,t_1)$. Here, we can choose time partition such that $t_1^-$, which denotes the endpoint of the interval $[0,t_1)$ and is distinguished from the start point of the interval $[t_1,t_2)$, is not in negligible set.  Immediately, we have boundedness of the second moment of $\rho^{(p)}(\dots,t_1)$ as follows
\begin{equation*}
\begin{split}
&\iint\limits_{\ \RR^2\times\RR^2}\cdots \iint\limits_{\ \RR^2\times\RR^2} \sum_{i=1}^p(|x_i|^2+|v_i|^2)\rho^{(p)}(\cdots,t_1)\rd x_1\rd v_1\ldots\rd x_p\rd v_p \\
=& \sum_{i=1}^p \iint\limits_{\RR^{2}\times\RR^2} (|x_i|^2+|v_i|^2) \tilde\mu(x_i,v_i,t_1)\rd x_i \rd v_i
=p\iint\limits_{\RR^{2}\times\RR^2} (|x|^2+|v|^2) \tilde\mu(x,v,t_1)\rd x \rd v\\
=&p\iint\limits_{\RR^{2}\times\RR^2} (|x|^2+|v|^2) \rd x\rd v\iint\limits_{\ \RR^2\times\RR^2}\cdots \iint\limits_{\ \RR^2\times\RR^2} \rho^{(p)}(x,x_2,\cdots,x_p,v,v_2,\cdots,v_p,t^-_{1})\rd x_2\rd v_2\ldots\rd x_p\rd v_p\\
\leq&p\iint\limits_{\ \RR^2\times\RR^2}\cdots \iint\limits_{\ \RR^2\times\RR^2}\sum_{i=1}^p(|x_i|^2+|v_i|^2)\rho^{(p)}(\cdots,t_1^-)\rd x_1\rd v_1\ldots\rd x_p\rd v_p 
\leq pC(t_1,\mu_0),
\end{split}
\end{equation*}
and $\rho^{(p)}(t_1)\in L^1\cap L^\infty((\RR^{2}\times\RR^2)^p)$. Therefore, by applying again Theorem~\ref{th:energy} there exists a weak solution $\rho^{(p)}$ for $t\in[t_1,t_2)$. By induction, for any $T>0$ and for almost every time partition, there exists a weak solution for (\ref{eq:RBMmeanfield}) in sense of \cite{carrillo2018mean}. 

Next, we prove the energy inequality~(\ref{eq:RBMenergy}). We multiple both sides of~(\ref{eq:RBMmeanfield}) by $\sum_{i=1}^p ({1\over 2}|v_i|^2+V(x_i))$, then integrate it over $(x_1,v_1,\ldots,x_p,v_p)\in(\RR^{2}\times\RR^2)^p$. Taking an integration by part, we get
\begin{multline*}
{\rd\over\rd t}\iint\limits_{\ \RR^2\times\RR^2}\cdots \iint\limits_{\ \RR^2\times\RR^2}\sum_{i=1}^p\left(\cfrac{|v_i|^2}{2}+V(x_i)\right)\rho^{(p)}(\cdots,t)\rd x_1\rd v_1\ldots\rd x_p\rd v_p\\ 
\leq -\iint\limits_{\ \RR^2\times\RR^2}\cdots \iint\limits_{\ \RR^2\times\RR^2}\sum_{i=1}^p (\sigma+\Omega_i^{Im})|v_i|^2\rho^{(p)}(\cdots,t)\rd x_1\rd v_1\ldots\rd x_p\rd v_p.
\end{multline*}

On the one hand, since the distribution $\rho^{(p)}$ for $p$-tuple $\{(x_i,v_i)\}_{i=1}^p$ is symmetric , we have
\begin{equation*}
\begin{split}
&{\rd\over\rd t}\iint\limits_{\ \RR^2\times\RR^2}\cdots \iint\limits_{\ \RR^2\times\RR^2}\sum_{i=1}^p\left(\cfrac{|v_i|^2}{2}+V(x_i)\right)\rho^{(p)}(\cdots,t)\rd x_1\rd v_1\ldots\rd x_p\rd v_p\\
=&{\rd\over\rd t}\sum_{i=1}^p\iint\limits_{\RR^{2}\times\RR^2}\left(\cfrac{|v_i|^2}{2}+V(x_i)\right)\tilde{\mu}(x_i,v_i,t)\rd x_i\rd v_i
=p\cdot {\rd\over\rd t}\iint\limits_{\RR^{2}\times\RR^2}\left(\cfrac{|v|^2}{2}+V(x)\right)\tilde{\mu}(x,v,t)\rd x\rd v.
\end{split}
\end{equation*}
On the other hand, for the right-hand side, we have
\begin{equation*}
\begin{split}
&-\iint\limits_{\ \RR^2\times\RR^2}\cdots \iint\limits_{\ \RR^2\times\RR^2}\sum_{i=1}^p (\sigma+\Omega_i^{Im}) |v_i|^2\rho^{(p)}(\cdots,t)\rd x_1\rd v_1\ldots\rd x_p\rd v_p\\
=&-\sum_{i=1}^p\iint\limits_{\ \RR^2\times\RR^2}\cdots \iint\limits_{\ \RR^2\times\RR^2} (\sigma+\Omega_i^{Im}) |v_i|^2\rho^{(p)}(\cdots,t)\rd x_1\rd v_1\ldots\rd x_p\rd v_p\\
=&-p\sigma\iint\limits_{\RR^{2}\times\RR^2} |v|^2\tilde{\mu}(x,v,t)\rd x\rd v - p \iint\limits_{\ \RR^2\times\RR^2}\cdots \iint\limits_{\ \RR^2\times\RR^2} \Omega_1^{Im} |v_1|^2\rho^{(p)}(\cdots,t)\rd x_1\rd v_1\ldots\rd x_p\rd v_p.
\end{split}
\end{equation*}
Therefore, we achieve the proof. 

\end{proof}

\subsection{The hybrid model}
Though the RBM algorithm reduces the computational complexity from $\mathcal{O}(N^2)$ to $\mathcal{O}(N)$, it has some limitations since each particle can only interact with $p-1$ particles within the same batch in each time step.  For instance, it may happen that particles having high collision risk may not be assigned into one batch, especially when $N\gg p$. Some particles may neglect their surrounding particles since they are not in the same batch. Thus, they might collide or even overlap  other particles. To avoid these cases, a natural idea is to add  surrounding particles into the "vision cone" of the particle, therefore particle can predict all potential collisions. 

To this end, we combine a so-called Cell-List approach\cite{frenkel2001understanding}, where  the whole space is divided into cells of equal size. One particle in a given cell interacts with those particles in the same and neighboring cells and with those in the same batch (see Figure~\ref{fig:list}). We call the interaction within the same and neighboring cell for one particle the short-range force. Since each particle has a fixed volume, the number of particles in a cell has a upper bound. So the number of particles in a cell is independent of the system size $N$ when $N$ is large.

Here, we take the same time partition and batch division as before. In addition, the whole space is divided into cells with size $r_c\times r_c$, so each particle belongs to a cell.
We assume that the particle $i$, for $\{1,\ldots,N\}$, belongs to the batch $C_{q}$. Let us denote the set $\tilde{C}_i$ containing all the particles in the same cell or nearby cells for the particle $i$. Therefore, in one time step, the particle $i$ will only interact with those particles in the set $C_{q}\cup\tilde{C}_i$. 

In perception phase, the three interaction sets of the hybrid model are determined as follows:
\begin{gather}
\mathcal{K}^{Co,HM}_i := \mathcal{K}^{Co}_i(t) \cap (C_{q}\cup \tilde{C}_i),\notag\\
\mathcal{K}^{Im,HM}_i := \mathcal{K}^{Im}_i(t) \cap (C_{q}\cup \tilde{C}_i),\notag\\
\mathcal{K}^{Fo,HM}_i := \mathcal{K}^{Fo}_i(t) \cap (C_{q}\cup \tilde{C}_i),\notag
\end{gather}
where the definitions of $\mathcal{K}^{Co}_i,\mathcal{K}^{Im}_i,\mathcal{K}^{Fo}_i$ can be referred in~\eqref{eq:KCo}, \eqref{eq:KIm}, \eqref{eq:KFo} respectively, and HM stands for the hybrid model.

\begin{figure}
\centering
\includegraphics[width=0.3\linewidth]{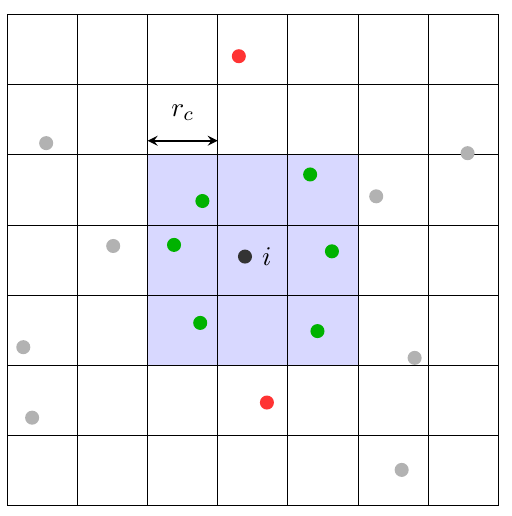}
%
%
\caption{\label{fig:list}Illustration of  particles interacting with the particle $i$. The space is divided into cells of size $r_c\times r_c$; the particle $i$ can interact with those particles (green particles) in the same and nearby cells (in purple cells) and with those particles (red particles) in the same batch in the hybrid model.}
\end{figure}

In decision-making phase, the corresponding three interaction forces are determined as follows:
\begin{align*}
F^{Co,HM}_i(t) &= \cfrac{1}{\overline{\mathcal{K}^{Co,HM}_i}+\beta}\sum_{j\in C_{q},j\neq i} \omega^{Co}_{ij} \mathbbm{1}_{\mathcal{K}^{Co,HM}_i}(j) v^\bot_i, \\
F^{Im,HM}_i(t) &=-\cfrac{1}{\overline{\mathcal{K}^{Im,HM}_i}+\beta}\sum_{j\in C_{q},j\neq i} \omega^{Im}_{ij} \mathbbm{1}_{\mathcal{K}^{Im,HM}_i}(j) v_i, \\
F^{Fo,HM}_i(t) &=\cfrac{1}{\overline{\mathcal{K}^{Fo,HM}_i}+\beta}\sum_{j\in C_{q},j\neq i} \omega^{Fo}_{ij} \mathbbm{1}_{\mathcal{K}^{Fo,HM}_i}(j) v^\bot_i.
\end{align*}

Thus, in one time step, the algorithm contains three steps:
\begin{enumerate}
\item[(i)] Divide particles into $n$ batches randomly;

\item[(ii)] Determine which cell one particle belonging to;

\item[(iii)]Evolve interactions within batches or within the same and nearby cells.
\end{enumerate}
The hybrid model for collision avoidance obeys the rules as follows: 
\begin{equation}\label{eq:ls}
\begin{cases}
\cfrac{\rd x_i}{\rd t} = v_i, \\[3mm]
\cfrac{\rd v_i}{\rd t} = F^{HM}_i(t):= F^{Co,HM}_i(t)+F^{Im,HM}_i(t)+F^{Fo,HM}_i(t) + F^{Ex}_i(t),
\end{cases}\qquad i\in\{1,\ldots,N\}.
\end{equation}

Since the number of particles in the cell has an upper bound and the batch size $p$ is independent of the system size $N$, the computational complexity of the hybrid model remains to be of order $\mathcal{O}(N)$ when $N\to\infty$. The detailed algorithm is shown in Algorithm~\ref{ta:hRBM}.

\begin{algorithm}
\caption{\label{ta:hRBM}The hybrid model for collision avoidance}
\begin{algorithmic}[1]
\State Given a time partition $t_k,k\in\{0,1,\cdots,m\}$.
\For{$k$ in $1:m$}
\State Determine which cell each particle belonging to.
\State Divide $\{1,2,\ldots,N\}$ into $n=N/p$ batches randomly.
\For{each batch $C_q$}
\For{each particle $i\in C_q$}
\State Compute the set $\tilde{C}_i$.
\State Update $(x_i,v_i)$ by solving the following ODEs with $t\in[t_{k-1},t_k)$.
\Statex\refstepcounter{equation}\thetag{\ref{eq:ls}}$\hfill\begin{cases}
\displaystyle\cfrac{\rd x_i}{\rd t} = v_i, \\[3mm]
\displaystyle\cfrac{\rd v_i}{\rd t} = F^{HM}_i(t).
\end{cases}\hfill\hfill$
\EndFor
\EndFor
\EndFor
\end{algorithmic}
\end{algorithm}

\section{Numerical results}\label{sec:Numericalresults}
\subsection{Setup}
The common value of parameters used in the numerical simulations are listed in the Table~\ref{ta:para}. These values are determined by fitting the model with a lot of experimental data. In the case where the value of $\varepsilon$ in the mollification operator $\eta_\varepsilon$ defined in (\ref{eq:H_ij}) and (\ref{eq:sinus}) is small enough, we can just substitute $\cos^\varepsilon(\alpha_{ij}), \sin^\varepsilon(2\alpha_{ij})$ by $\cos(\alpha_{ij}), \sin(2\alpha_{ij})$ in the numerical simulations.
\begin{table}
\caption{\label{ta:para}The value of parameters of the collision-avoidance model.}
\centering
\begin{tabular}{|c c c | c c c|}
\hline
Variable & Value & Reference & Variable & Value & Reference\\
\hline
$\kappa$ & $0.5$ & Definition~\ref{def:visioncone} & $R_0^{Im}$ & $1.0$ & Eq.(\ref{eq:KIm})\\
$R_0$ & $0.5$ & Remark~\ref{remark:saferadius} & $R^{Im}$ & $3.0$ & Eq.(\ref{eq:KIm})\\
$R$ & $3.0$ & Remark~\ref{remark:saferadius} & $C_2$ & $e$ & Eq.(\ref{eq:omegaim})\\
$C_0$ & $6\pi$ & Eq.(\ref{eq:omegaCo}) & $C_3$ & $1.0$ & Eq.(\ref{eq:omegaim})\\
$C_1$ & $2$ & Eq.(\ref{eq:omegaCo}) & $R^{Fo}$ & $3.0$ & Eq.(\ref{eq:KFo})\\
$\delta_0$ & $0.01$ & Eq.\eqref{eq:sign} & $C_4$ & $\pi$ & Eq.(\ref{eq:omegafo})\\
$\delta_1$ & $0.1$ & Eq.\eqref{eq:sign} & $C_5$ & $1.0$ & Eq.(\ref{eq:omegafo})\\
$\beta$ & $0.01$ & Eq.(\ref{eq:fco}) & $\sigma$ & $1.0$ & Eq.(\ref{eq:fex})\\
\hline
\end{tabular}
\vskip .cm
\end{table}
To lead each particle to its destination, the potential function $V(x)$ in (\ref{eq:fex}) is set as follows:
\begin{equation*}
V(x) = \|x-x_T\|.
\end{equation*}

For simplicity, we call (\ref{eq:modelODE}) \textbf{the original model} where each particle needs to interact with all others and the models in (\ref{eq:RBM}) and (\ref{eq:ls}) are called \textbf{the RBM model} and \textbf{the hybrid model} respectively. In addition, the model with interactions only in the cell list (the situation when the batchsize $p$ in the hybrid model (\ref{eq:ls}) is set as $1$) is called \textbf{the short-force model}. The main objective of our numerical simulations is to compare these models. We use the improved Euler's method as the numerical scheme, and denote the time step by $\Delta t$. The time to shuffle the batches for the RBM model or the hybrid model is set as $\{t_k:t_k=k\Delta t\}_{k=0}^\infty$. In order to reduce the error of energy in the simulations, the velocity $v_i$ in ODEs is illustrated in polar coordinate and a corresponding numerical scheme is constructed.

As illustrated in Remark~\ref{remark:saferadius}, a collision occurs when the distance between two particles is less than $2R_0$. It is assumed that the collision is an inelastic one with coefficient of elasticity $e_c=0.8$, i.e.
\begin{equation*}
e_c=\cfrac{\|v_i-v_j\|}{\|v_{i0}-v_{j0}\|}=0.8,
\end{equation*}
where $v_{i0},v_{j0}$ are the velocities projected to the line connecting the center of the particle $i$ and the particle $j$ before collision, while $v_{i},v_{j}$ are the ones after collision respectively(see Figure~\ref{fig:collision}).

To show the results clearly, each particle is set as a disk with radius $R_0$, and the trajectory during past $5$s is illustrated by semitransparent line. The velocity is also marked out as a thin line with one side located at the position of each particle. The sketch can be seen in Figure~\ref{fig:particle}.

\begin{figure}
\centering
\begin{tabular}{c c c}
\begin{tikzpicture}[
par/.style={circle,inner sep=0,minimum size=8mm,fill=black!60,draw=black,very thick},
ppar/.style={par,fill=white},>=stealth,shorten >=1pt]
\node (i) at (0,0) [ppar]{};
\node (j) at (1.,0) [par]{};
\draw[->, thick] ($(i)+(-.5,.6)$)--node[above]{$v_{i0}$}($(i)+(.5,.6)$);
\draw[->, thick] ($(j)+(-.3,.6)$)--node[above]{$v_{j0}$}($(j)+(.3,.6)$);
\draw[dashed] (-.7,0)--(1.7,0);
\end{tikzpicture} &
\begin{tikzpicture}[
par/.style={circle,inner sep=0,minimum size=8mm,fill=black!60,draw=black,very thick},
ppar/.style={par,fill=white},>=stealth,shorten >=1pt]
\node at (0,0) [ppar]{};
\node at (.8,0) [par]{};
\draw[->, thick] (0,0)+(-.4,.0)--node[left]{$-F$}($(0,0)+(-.8,.0)$);
\draw[->, thick] ($(.8,0)+(.4,.0)$)--node[right]{$F$}($(.8,0)+(.8,.0)$);
\draw[dashed] (-.7,0)--(1.5,0);
\end{tikzpicture} &
\begin{tikzpicture}[
par/.style={circle,inner sep=0,minimum size=8mm,fill=black!60,draw=black,very thick},
ppar/.style={par,fill=white},>=stealth,shorten >=1pt]
\node (i) at (0,0) [ppar]{};
\node (j) at (1.,0) [par]{};
\draw[->, thick] ($(i)+(-.2,.6)$)--node[above]{$v_{i}$}($(i)+(.2,.6)$);
\draw[->, thick] ($(j)+(-.4,.6)$)--node[above]{$v_{j}$}($(j)+(.4,.6)$);
\draw[dashed] (-.7,0)--(1.7,0);
\end{tikzpicture}\\
(a) before collision &
(b) collision happen &
(c) after collision
\end{tabular}
\caption{\label{fig:collision}The sketch of a collision between the particle $i$ and the particle $j$. The vertical component of the velocity is omitted because it remains unchanged by the collision.}
\end{figure}
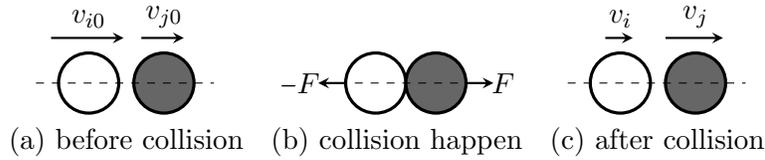

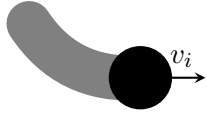
\begin{figure}
\centering
\begin{tikzpicture}[
par/.style={circle,inner sep=0,minimum size=8mm,fill=black,draw=black,very thick},
ppar/.style={par,fill=white},>=stealth,shorten >=1pt,scale=.5]
\draw[opacity=0.5,line width=6mm,line cap=round] ($(0,0)$)to[bend left=30]($(0,0)+(-3.,1.5)$);
\node (i) at (0,0) [par]{};
\draw[->, thick] ($(i)+(.4,.0)$)--node[above]{$v_{i}$}($(i)+(1.8,.0)$);
\end{tikzpicture}
\caption{\label{fig:particle}The illustration of the motion of the particle $i$ in the numerical results.}
\end{figure}
\subsection{Ability of the model to avoid collisions}
In this part, we present numerical results corresponding to four classical examples. The first three examples show the ability of the original model to avoid  collisions, while the last example shows the necessity for our model to add a deceleration force and a following interaction force. Moreover, a comparison among the three models defined in \eqref{eq:modelODE},\eqref{eq:RBM},\eqref{eq:ls} is provided in the last example to show the advantages of the hybrid model. It is noted that the time step $\Delta t$ is set as $2^{-7}$ in these four examples and the width of the cell list in the hybrid model and the short-force model is set as $4$m.

The first example is called \textbf{Circle} and the results are shown in Figure~\ref{fig:cir4p}. In this example, four particles are initially located along a circle and each particle's destination is the diametrically opposed position. The diameter of the circle is set as $10$m. Each particle would go ahead to its goal through the center of the circle if there is no interaction with others. It may happen that all the particles in the example go ahead to the center immediately and get stuck  there. But actually, each particle turn its direction in time and collisions are avoided.
\begin{figure}
\centering
\begin{subfigure}{0.25\linewidth}
\centering
\includegraphics[width=\linewidth]{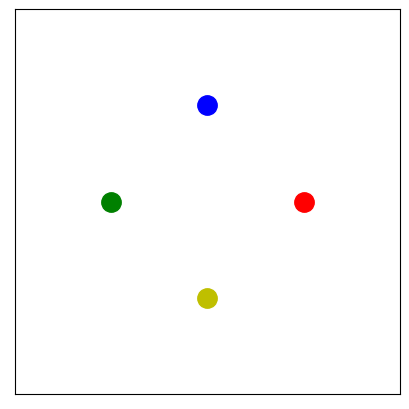}
\caption{$t=0s$}
\end{subfigure}
\begin{subfigure}{0.25\linewidth}
\centering
\includegraphics[width=\linewidth]{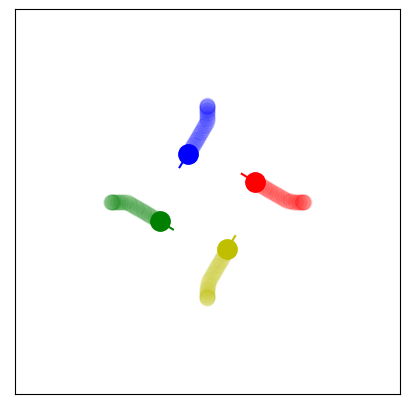}
\caption{$t=4s$}
\end{subfigure}

\begin{subfigure}{0.25\linewidth}
\centering
\includegraphics[width=\linewidth]{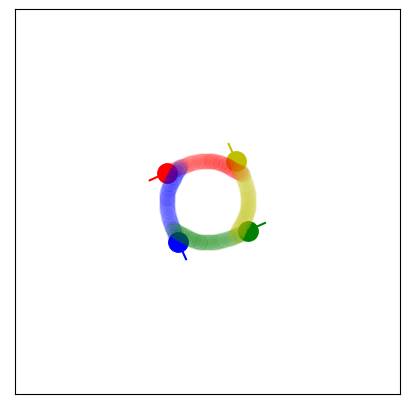}
\caption{$t=10s$}
\end{subfigure}
\begin{subfigure}{0.25\linewidth}
\centering
\includegraphics[width=\linewidth]{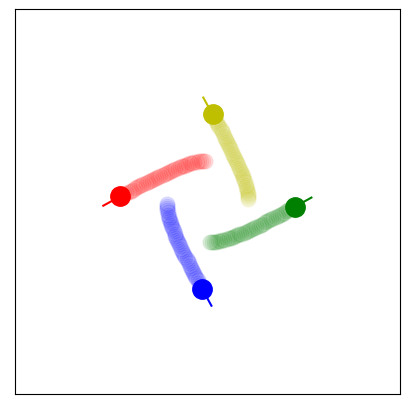}
\caption{$t=13s$}
\end{subfigure}
\caption{\label{fig:cir4p}\textbf{Circle:} four particles are initially located along a circle with $10$m diameter symmetrically. Each one's destination is the opposite position on the circle. The solution of the original model at different times is shown.}
\end{figure}

The second example is called \textbf{Obstacles} and the results are shown in Figure~\ref{fig:conDF2o}. In this example, twenty particles are initially located as four lines besides two circle obstacles with initial velocity $0.7$m/s facing the obstacles and each one's destination is the point $60$m in front of its initial position. The spacing between particles is $2$m in row and $3$m in line. The diameter of two circle obstacles are $4$m and $8$m respectively. We demonstrate the ability of our original model to avoid collisions with the static obstacles. At the beginning, the ones near the obstacles would decelerate to avoid collision and the ones on both sides would deviate to move around the obstacles. As time goes by, all the particles would line up in two lines and move around the obstacles to their destinations. The collision-interaction force, the imminent-interaction force and the following-interaction force collaborate nice to attain the aim in this example.
\begin{figure}
\centering
\begin{tabular}{c c c c}
\includegraphics[width=.18\linewidth]{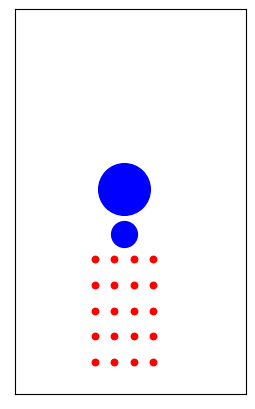}&
\includegraphics[width=.18\linewidth]{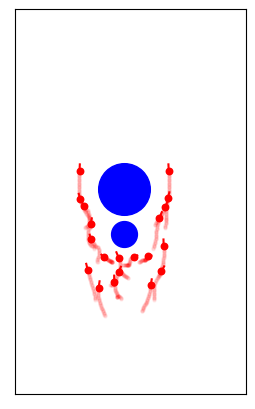}&
\includegraphics[width=.18\linewidth]{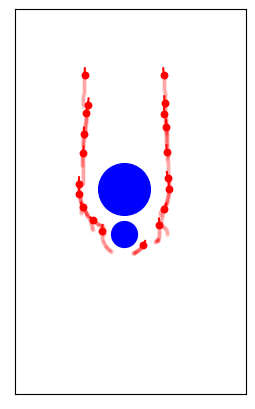} &
\includegraphics[width=.18\linewidth]{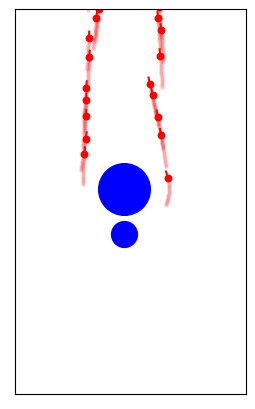}\\
(a) $t=0s$ &
(b) $t=15s$ &
(c) $t=30s$ &
(d) $t=45s$
\end{tabular}
\caption{\label{fig:conDF2o}\textbf{Obstacles:} Twenty particles are initially located in front of two circle obstacles with initial velocity $0.7$m/s facing the obstacles. The solution of the original model at different times is shown.}
\end{figure}

The third example is called \textbf{Crossing} and the results are shown in Figure~\ref{fig:conDF2t}. In this example, two groups of particles encounter at the intersection of two orthogonal lanes. Each group has twenty-five particles with spacing $1$m in row and $3$m in line. Each particle's destination is the point $44$m in front of its initial position and each has an initial velocity $0.7$m/s towards its goal. The width of two lanes is $20$m. It may happen that two groups of particle get congested in the intersection. Actually, the original model can avoid this situation to happen. When two groups meet, some particles would decelerate and some would deviate. It can be seen that, particles are split in several subgroups and pass through the crossroad in turn. 

\begin{figure}
\centering
\begin{subfigure}{0.25\linewidth}
\centering
\includegraphics[width=\linewidth]{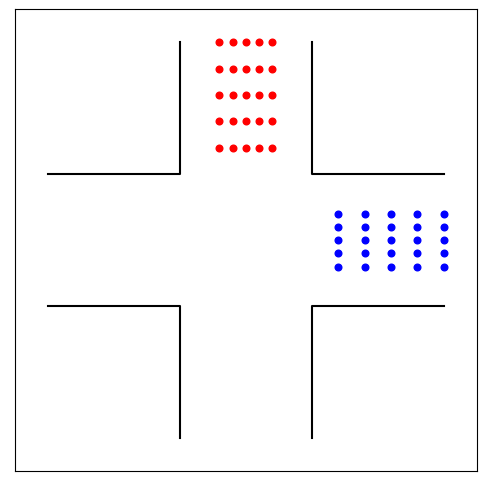}
\caption{$t=0s$}
\end{subfigure}
\begin{subfigure}{0.25\linewidth}
\centering
\includegraphics[width=\linewidth]{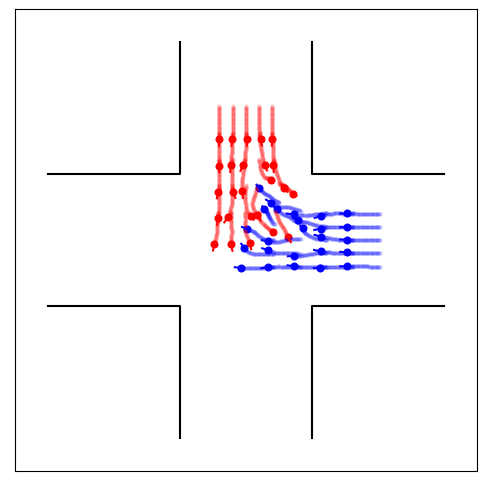}
\caption{$t=15s$}
\end{subfigure}

\begin{subfigure}{0.25\linewidth}
\centering
\includegraphics[width=\linewidth]{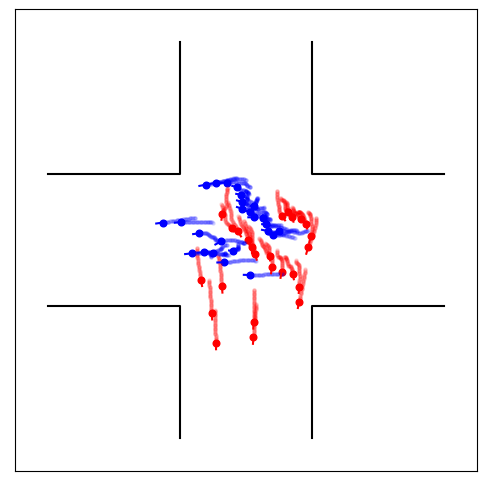}
\caption{$t=30s$}
\end{subfigure}
\begin{subfigure}{0.25\linewidth}
\centering
\includegraphics[width=\linewidth]{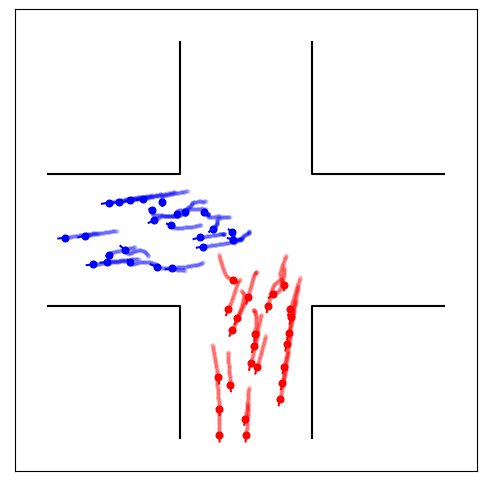}
\caption{$t=45s$}
\end{subfigure}
\caption{\label{fig:conDF2t}\textbf{Crossing:} Two groups meeting at the intersection of two orthogonal lanes. The solution of the original model at different time is shown.}
\end{figure}

The forth example is called \textbf{Group-swap} and the results are shown in Figure~\ref{fig:con100p1} and~\ref{fig:con100p2}. In this example, twenty-four particles are initially separated in two groups. The goal is to swap these two groups. The spacing between particles is $0.8$m in row and $2$m in line and each particle has an initial velocity $1$m/s (see Figure~\ref{fig:con100p1}(a)). The difficulty in this example is that the particles may not be able to decelerate or deviate in time to avoid collisions and may get congested.

Figure~\ref{fig:con100p1} (b) and (c) shows the necessity of introducing the imminent-interaction force and the following-interaction force. Without the imminent-interaction and the following-interaction force ($C_2=C_4=0$), particles can only avoid collisions by turning the motion direction. In Figure~\ref{fig:con100p1} (c), one red particle marked out by a black circle collides with two blue particles since the particles are too dense to have enough space to deviate. On the contrary, in Figure~\ref{fig:con100p1} (b), both the red particle and the blue one circled decelerate in time and move around each other with low velocity to pass by successfully. It is noted that the particles with the same direction line up to move in Figure~\ref{fig:con100p1} (b) while the particles in Figure~\ref{fig:con100p1} (c) do not. So we remark that lining up can relieve congestion when particles are dense.


\begin{figure}
\centering
\begin{subfigure}[t]{0.2\linewidth}
\includegraphics[width=\linewidth]{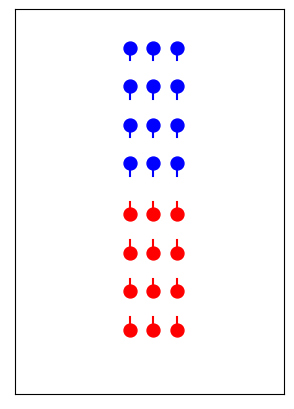}
\caption{}
\end{subfigure}
\begin{subfigure}[t]{0.2\linewidth}
\includegraphics[width=\linewidth]{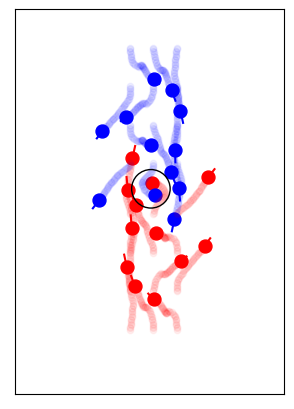}
\caption{}
\end{subfigure}
\begin{subfigure}[t]{0.2\linewidth}
\includegraphics[width=\linewidth]{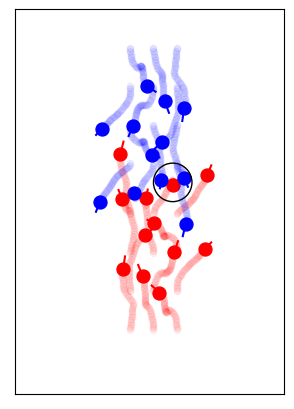}
\caption{}
\end{subfigure}
\caption{\label{fig:con100p1}\textbf{Group-swap:} twenty-four particles are split into two groups to swap with each other. (a) the initial state of particles; (b) the solution of the original model at $t=5$s; (c) the solution of the model without imminent and following interactions at $t=5$s.}
\end{figure}

We also compare the performance of the different models in this example and the results are shown in Figure~\ref{fig:con100p2}. The solution of the original model in Figure~\ref{fig:con100p2} (a) shows that the particles can line up to cross through each other to the destinations without collision or congestion, and the hybrid model has the same good performance in Figure~\ref{fig:con100p2} (b). However, it can happen that some particles would overlap in the RBM model. Because the interaction between particles may insufficient, especially when the number of particles are large, some collisions would not be detected and particles would overlap when they are in different batches in RBM model. We mark out an example of overlapping particles by a black circle in Figure~\ref{fig:con100p2} (c). Thus, the Cell-List approach in the hybrid model is necessary to avoid overlapping and collisions. Meanwhile, the result of the short-force model in Figure~\ref{fig:con100p2} (d) shows the necessity of the random batch in the hybrid model. With interaction only in the cell list, each particle can only detect ones surrounding it. So particles without the knowledge of the overall situation do not line up in time, which results in unavoidable collision and congestion in Figure~\ref{fig:con100p2} (d).

In conclusion, the original model and the hybrid model both have a good ability to avoid collisions and relieve congestion. Both the RBM algorithm and the Cell-List approach are necessary in the hybrid model. Indeed, on one hand, without the Cell-List approach, particles may overlap others, and on the other hand, without RBM algorithm, particles lack the knowledge of the overall situation and may not line up to avoid congestion in time.

\begin{figure}
\centering
\begin{subfigure}[t]{0.25\textwidth}
\centering
\includegraphics[width=.8\textwidth]{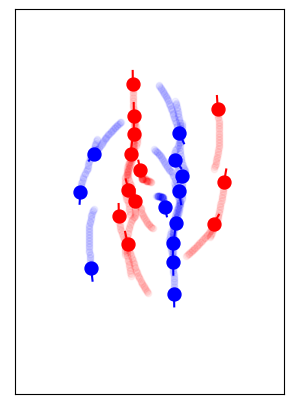}
\caption{The original model}
\end{subfigure}
\begin{subfigure}[t]{0.25\textwidth}
\centering
\includegraphics[width=.8\textwidth]{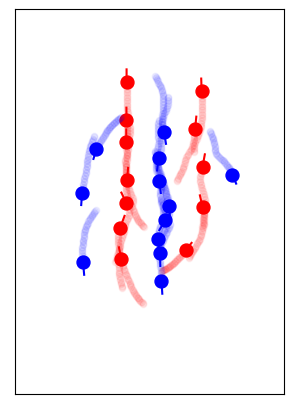}
\caption{The hybrid model}
\end{subfigure}

\begin{subfigure}[t]{0.25\textwidth}
\centering
\includegraphics[width=.8\textwidth]{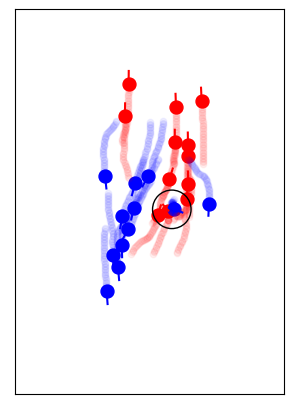}
\caption{The RBM model}
\end{subfigure}
\begin{subfigure}[t]{0.25\textwidth}
\centering
\includegraphics[width=.8\textwidth]{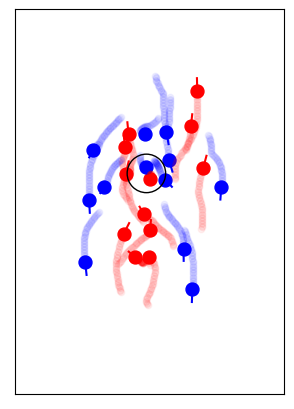}
\caption{The short-force model}
\end{subfigure}
\caption{\label{fig:con100p2}\textbf{Group-swap:} the solution of different models at $t=11$s.}
\end{figure}

\subsection{Energy loss and $L^2$-norm of the model}\label{sec:energynorm}
In this subsection, we compare the hybrid model with the RBM model by an qualitative analysis thanks to the energy loss and the $L^2$-norm. Our objective is to get the solution of trajectory for all particles without collisions or overlapping. So, considering this objective, the total energy loss caused by inelastic collisions between particles should be as little as possible. In addition, spacing between each two particles should not be too small otherwise collisions or overlapping may happen. To evaluate whether overlapping of particles happens, we substitute the distribution function $\delta(x-x_i)$ for any particle $i$ by a Gaussian distribution function and define a smooth distribution function $g^N(x,t)$, which is a normalized superposition of the $N$ particles' Gaussian distribution functions in sense of $L^1$-norm, defined by
\begin{equation}\label{eq:gN}
g^N(x,t) = {1\over N}\sum_{i=1}^N \left(\cfrac{1}{\pi a^2}\right) e^{-\|x-x_i\|^2\over a^2}, \quad \forall x\in\RR^2,
\end{equation}
where $x_i=x_i(t), \forall i\in \{1,\ldots,N\}$. And the $L^2$-norm of the function $g^N(x,t)$ is given by
\begin{equation}\label{eq:l2norm}
\|g^N\|_{L^2} = \left(\int\limits_{\RR^2} {g^N(x,t)}^2\rd x\right)^{1\over2} =\left({1\over 2\pi a^2N^2} \sum_{i,j} e^{-\|x_i-x_j\|^2\over2a^2}\right)^{1\over 2}.
\end{equation}

With proper parameter $a$, the $L^2$-norm of $g^N(x,t)$ can present how many particles overlap. For the two particles $i$ and $j$, it is required that the corresponding exponential term in \eqref{eq:l2norm} should grow to $1$ only if overlapping happens ($\|x_i-x_j\|<2R_0$), while the term should maintain small if the two particles are close but without overlapping. So it is important to choose the proper value of $a$ in~(\ref{eq:gN}) to reflect the overlapping phenomenon accurately. The parameter $a$ is determined by the relationship below,
\begin{equation*}
\int\limits_{\mathcal{B}(0,R_0)}\left(\cfrac{1}{\pi a^2}\right) e^{-\|x\|^2\over a^2}\rd x = 0.99,
\end{equation*}
where $\mathcal{B}(0,R_0)=\{x\in\RR^2:\|x\|<R_0\}$, $R_0=0.5$. It means that the distribution is mainly included in the volume disk of the particle, satisfying the requirement that only overlapping can result in rapidly growth of the exponential term in \eqref{eq:l2norm}. Therefore,  the value of parameter $a$ is computed as $\frac{0.5}{\sqrt{2\ln(10)}}$.

In summary, two estimates are introduced to analyze the performance of the two models. One is the energy loss caused by inelastic collisions, which evaluate the ability for the model to avoid collisions. Another is the $L^2$-norm of the distribution function $g^N(x,t)$, reflecting whether particles overlap others.
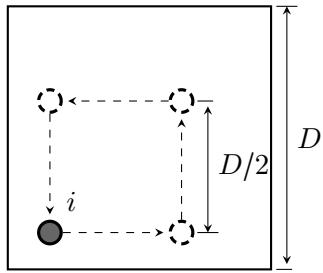
\begin{figure}
\centering
\begin{tikzpicture}[
par/.style={circle,inner sep=0,minimum size=3mm,fill=black!60,draw=black,very thick},
ppar/.style={par,fill=white,densely dashed},>=stealth,shorten >=2pt, scale=0.7]
\draw[thick] (0,0) rectangle (5,5);
\node (i) at (.8,.7)[par,label=60:$i$]{};
\node (i1) at ($(i)+(2.5,0)$)[ppar]{};
\node (i2) at ($(i1)+(0,2.5)$)[ppar]{};
\node (i3) at ($(i2)+(-2.5,0)$)[ppar]{};
\draw[->,dashed] (i)to(i1);
\draw[->,dashed] (i1)to(i2);
\draw[->,dashed] (i2)to(i3);
\draw[->,dashed] (i3)to(i);
\draw (5.1,5)--(5.5,5) (5.1,0)--(5.5,0);
\draw[<->] (5.3,5)--node[right]{$D$}(5.3,0);
\draw ($(i1)+(.3,0)$)--($(i1)+(.7,0)$) ($(i2)+(.3,0)$)--($(i2)+(.7,0)$);
\draw[<->] ($(i1)+(.5,0)$)--node[right]{$D/2$}($(i2)+(.5,0)$);
\end{tikzpicture}
\caption{\label{fig:500example}$N$ particles are distributed randomly in the square with width $D$ and confined in the square. Each particle move along a square with width $D/2$ counterclockwise.}
\end{figure}

We compare the hybrid model with the RBM model in an example illustrated in Figure~\ref{fig:500example}. In this example, all the particles are confined in a square with width $D$ and distributed randomly in this square. Each particle sets its destination to move along a square with width $D/2$ counterclockwise. Set $D=50$ and $N=500$, we solve the hybrid model and the RBM model during $t\in[0,T]$ under the time step $\Delta t=2^{-7},2^{-6},2^{-5},2^{-4}$ respectively. The corresponding results are shown in Figure~\ref{fig:energynorm}. Figure~\ref{fig:energynorm}(a) shows that the energy loss caused by collisions for the hybrid model accumulates as time goes and converges to zero as the time step $\Delta t\to 0$. In contrast, it is interesting to note that in Figure~\ref{fig:energynorm}(b), the energy loss for the RBM model diverges as the time step $\Delta t$ goes smaller. It is because interactions between the particles are insufficient in RBM model in these cases and the particles with high collision risk may overlap rather than collide since the particles are not in the same batch. As the time step $\Delta t$ getting smaller, the partition of batches is shuffled more frequently and the collisions are more likely to be detected, resulting in larger energy loss. 

Figure~\ref{fig:500example}(c) and (d) show the $L^2$-norm for the two models respectively. Without the overlapping phenomenon, the $L^2$-norm of the distribution function $g^N$ corresponding to the hybrid model maintains a constant with small perturbation while the one for the RBM model becomes large as time goes, reflecting the occurrence of overlapping.

In conclusion, the hybrid model performs better than the RBM model in this test case. The hybrid model has less energy loss and can avoid the overlapping phenomenon.
\begin{remark}
It is noted that as the time step $\Delta t\to 0$, with more frequent shuffle of batches, the energy loss for the RBM model would also converge to zero, and the $L^2$-norm would also be maintained as a constant similarly to the hybrid model but with a huge computational cost. Of course, it is far from our intention to use RBM algorithm to reduce the computational complexity. So, it's practical and useful to add the Cell-List approach to the RBM model to construct the hybrid model.
\end{remark}

\begin{figure}
\centering
\begin{tabular}{c c}
\includegraphics[width=.45\linewidth]{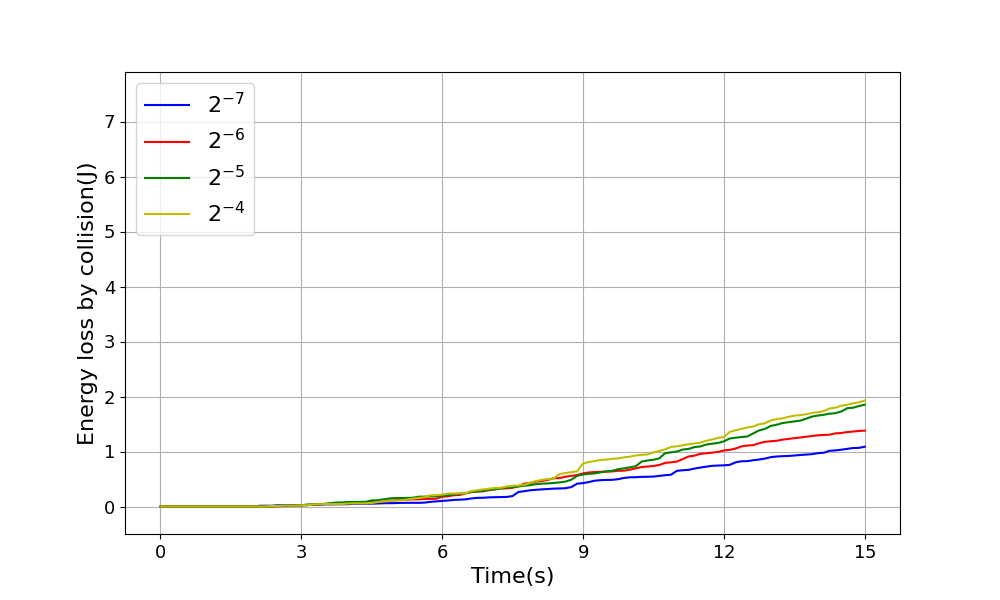}&
\includegraphics[width=.45\linewidth]{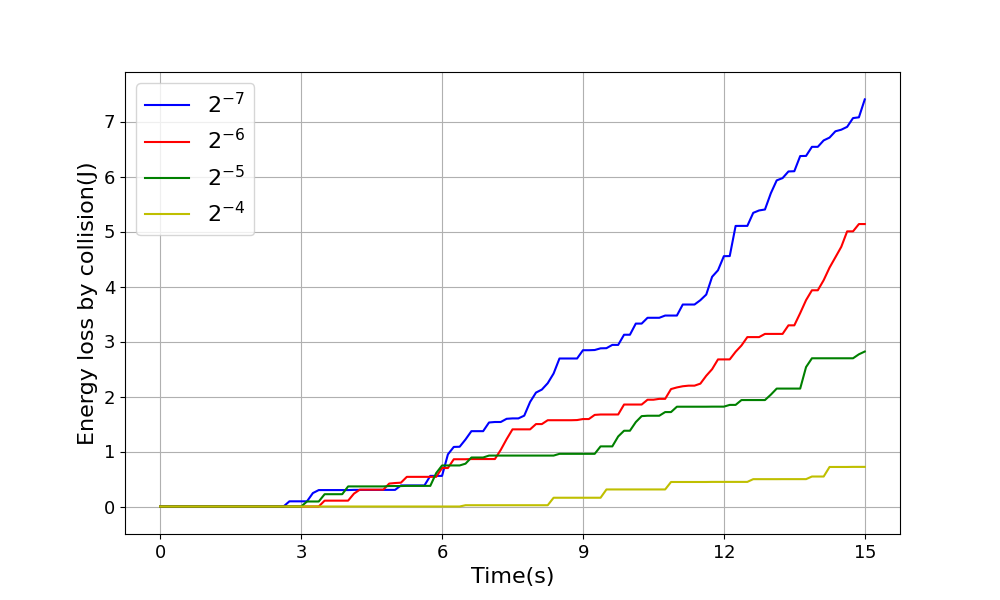}\\
(a) Energy loss for the hybrid model&
(b) Energy loss for the RBM model \\
\includegraphics[width=.45\linewidth]{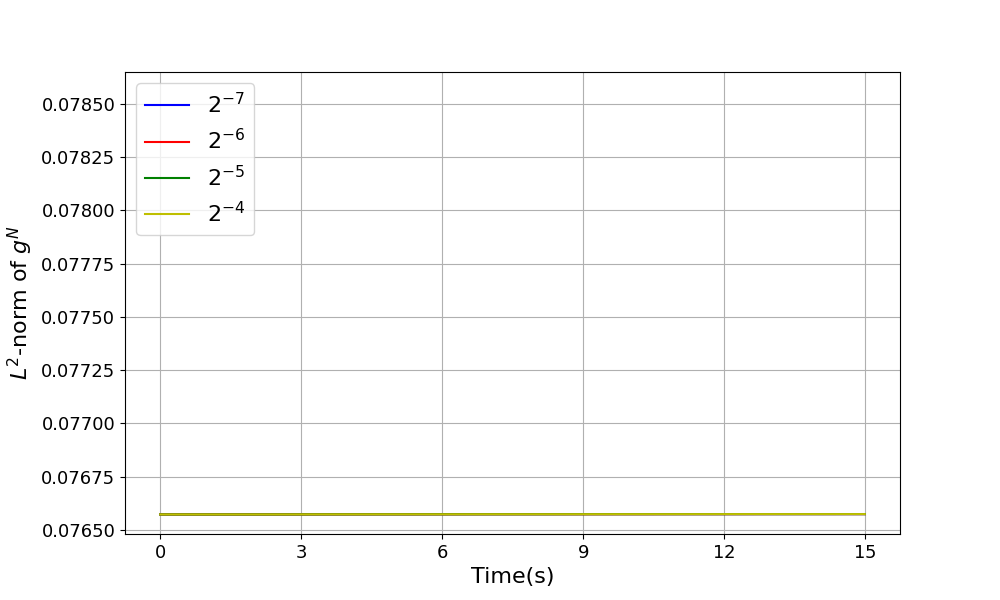}&
\includegraphics[width=.45\linewidth]{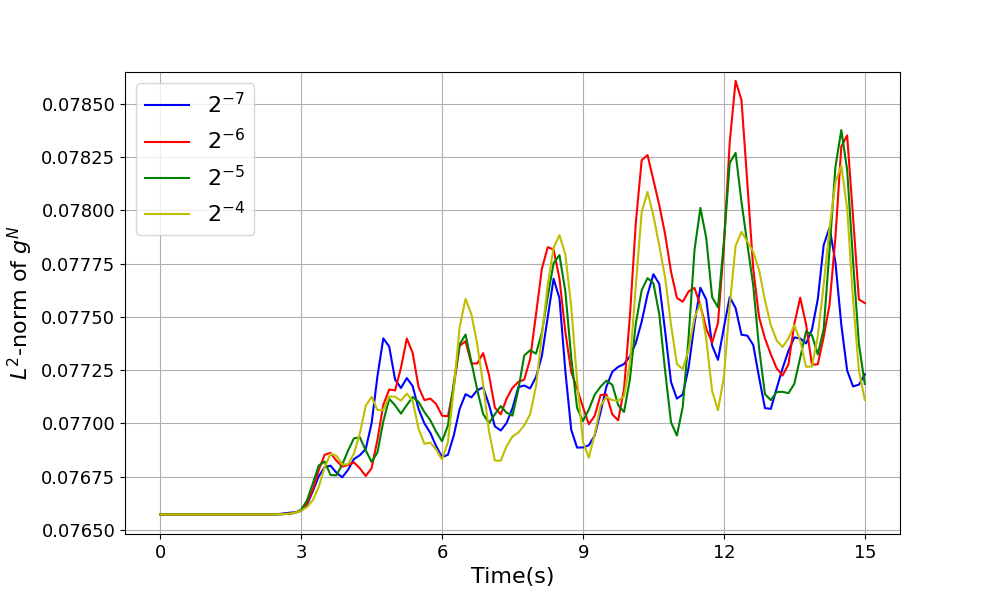}\\
(c) $L^2$-norm for the hybrid model&
(d) $L^2$-norm for the RBM model
\end{tabular}
\caption{\label{fig:energynorm}The energy loss and the $L^2$-norm for the hybrid model and the RBM model during $t\in[0,15]$ when $N=500$ and $D=50$.}
\end{figure}

\subsection{Efficiency of the hybrid model}
To test the practical performance for the hybrid model, we use the same example applied in Section~\ref{sec:energynorm} to show that the CPU time for the hybrid model and the original model as a function of the number of particles $N$. Since each model has a regular volume and the density of particles can not be too high, we prefixed the density of particles and test the two models in the situations with different number of particles. We choose $N=2^5,2^7,2^9,2^{11},2^{13}$, $D=25,50,100,200,400$ respectively and solve the two model during $t\in[0,15]$ with $\Delta t=2^{-4}$. The simulation was done using Python 3.8.10 on Linux system with one $2.6$ GHz Intel Core E5-2690v3 processor. The relationship between the simulation time and the number of particles are shown in loglog scale in Figure~\ref{fig:efficiency}. Clearly, the curve for the hybrid model is close to the straight line with slope $1$ and the one for the original model is close to the straight line with slope $2$. Thus, the computational complexity for the hybrid model is of order $\mathcal{O}(N)$ while the one for the original model is of order $\mathcal{O}(N^2)$. The results imply that our hybrid model is efficient to solve the collision-avoidance model in a reasonable time amount especially when the number of particles is huge.
\begin{figure}
\centering
\includegraphics[width=.5\linewidth]{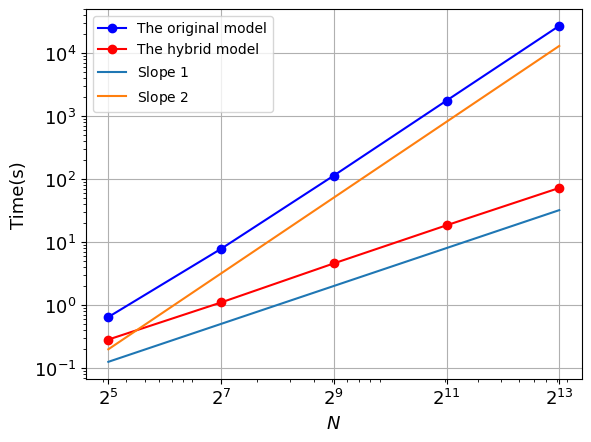}
\caption{\label{fig:efficiency}The simulation time of the hybrid model and the original model as a function of the number of particles $N$ in loglog scale.}
\end{figure}

\section{Conclusion and perspectives}\label{sec:Conclusion}
In this paper, we firstly present a novel two dimensional agent-based model for collision avoidance inspired from the former works for aerial vehicles and crowds~\cite{parzani2017three, ondvrej2010synthetic}. This agent-based model is a Newtonian system based on a ``vision cone'' framework. Compared to the former vision-based models, the particles can decelerate when the collision is imminent and can move in a line to relieve congestion apart from merely considering collision avoidance by deviations. We  also propose the corresponding mean-field limit model and prove the existence of the weak solution in sense of \cite{carrillo2018mean}.

Secondly, we introduce the RBM algorithm to solve the model efficiently and propose the corresponding mean-field limit for the RBM model. Combining the Cell-List approach with RBM, the hybrid model for collision avoidance is proposed for better performance to plan particles' motion flow without overlapping between particles. The hybrid model reduces the computational complexity from $\mathcal{O}(N^2)$ to $\mathcal{O}(N)$.

Thirdly, we perform several numerical simulations to illustrate the ability of the original model and the hybrid model to avoid collisions and relieve congestion. From the qualitative analysis of collision avoidance by the energy loss and the  $L^2$-norm conservation, we show that the hybrid model has a better performance compared to the RBM model. In addition, it is verified that the computational complexity of the hybrid model is of order $\mathcal{O}(N)$. In conclusion, the hybrid model has a good performance to collision avoidance with a low computational cost.
 
Our results are promising and open several future work directions. On one hand, the parameters in the model can be optimized automatically by some estimates to evaluate the performance of the model. On the other hand, the mean-field limit for the hybrid model would be investigated and proposed.
\section*{Acknowledgements}

This work has been supported by Heilongjiang Provincial Natural Science Foundation of China
(LH2019A013) and the Fundamental Research Funds for the Central Universities.

Zhichang Guo acknowledges support by the National Natural Science Foundation of China (12171123,
11971131, U21B2075).

%
  \appendix
\section{Proof of Theorem~\ref{th:lip}}
For the proof of the result stated in theorem~\ref{th:lip}  we split the force field $\mathcal{F}(f)$ into four parts: the collision-interaction force $\mathcal{F}^{Co}(f)$, the imminent-interaction force $\mathcal{F}^{Im}(f)$, the following-interaction force $\mathcal{F}^{Fo}(f)$ and the exit-interaction force $\mathcal{F}^{Ex}(f)$. We demonstrate these four parts are all Lipschitz continuous, then the Lipschitz continuity of the force field $\mathcal{F}(f)$ is deduced.

Let's recall that the total force is
\begin{equation*}
\mathcal{F}(f) = \mathcal{F}^{Co}(f) + \mathcal{F}^{Im}(f) + \mathcal{F}^{Fo}(f) +\mathcal{F}^{Ex}(f),
\end{equation*}
where
\begin{align}
\mathcal{F}^{Co}(f) &= \Omega_{Co}\cdot v^{\bot},\notag\\
\mathcal{F}^{Im}(f) &= -\Omega_{Im}\cdot v,\notag\\
\mathcal{F}^{Fo}(f) &= \Omega_{Fo}\cdot v^{\bot},\notag\\
\mathcal{F}^{Ex}(f) &= -\nabla_x V(x) - \sigma v. \notag
\end{align}

First, we prove thanks to the following proposition that the collision-interaction force $\mathcal{F}^{Co}(f)$ is Lipschitz continuous.
\begin{proposition}\label{pro:co}
For any $(x,v)\in\RR^2\times\RR^2 , t\in[0,T]$, we have following estimations,
\begin{equation*}
|\mathcal{F}^{Co}(f)(t,x,v)| \leq C\|v\| \|f\|_{L^1},
\end{equation*} and
\begin{equation*}
|\mathcal{F}^{Co}(f)(t,x,v)-\mathcal{F}^{Co}(f)(t,\tilde x,\tilde v)|\leq C(1+\|v\|)\left\|\begin{pmatrix}x\\v\end{pmatrix}-\begin{pmatrix}\tilde x\\ \tilde v\end{pmatrix}\right\|,  
\end{equation*}
where
\begin{equation*}
\begin{array}{r l}
\mathcal{F}^{Co}(f)(t,x,v) &= \Omega_{Co}\cdot v^{\bot}\\
 &= \cfrac{1}{\lambda^{Co}(t,x,v)} \iint\limits_{\RR^2\times\RR^2} m^{Co}(y-x,v,w)\mathbbm{1}_{\mathcal{K}^{Co}(v,w)}(y-x)f(t,y,w)\rd y\rd w \cdot v^\bot.
\end{array}
\end{equation*}
\end{proposition}

To better organize the proof of proposition\ref{pro:co}, we will introduce some notations, state some definitions and lemma.\\   
Let's denote 
\[\tilde{\mathcal{K}}^{Co}{(x,v)} = \{(y,w)\in\RR^2\times\RR^2 : (y-x)\in\mathcal{K}^{Co}{(v,w)}\},\] so we can rewrite $\mathbbm{1}_{\mathcal{K}^{Co}(v,w)}(y-x)=\mathbbm{1}_{\tilde{\mathcal{K}}^{Co}(x,v)}(y,w)$, and $\Omega_{Co}(x,v)$ as
\begin{equation*}
\begin{split}
\Omega_{Co}(x,v)&=\cfrac{1}{\lambda^{Co}(t,x,v)} \iint\limits_{\RR^2\times\RR^2} m^{Co}(y-x,v,w)\mathbbm{1}_{\mathcal{K}^{Co}(v,w)}(y-x)f(t,y,w)\rd y\rd w\\
&=\cfrac{1}{\lambda^{Co}(t,x,v)} \iint\limits_{\RR^2\times\RR^2} m^{Co}(y-x,v,w)\mathbbm{1}_{\tilde{\mathcal{K}}^{Co}(x,v)}(y,w)f(t,y,w)\rd y\rd w.
\end{split}
\end{equation*}

Let's introduce the notation
\begin{equation*}
A\Delta B = (A\backslash B)\bigcup (B\backslash A),
\end{equation*}
for any sets $A$ and $B$.
 
{Since $f$ is compactly support in phase space, we assume that $\exists  r>0$, such that $supp(f)\subset \{(x,v)\in\RR^2\times\RR^2 :\|x\|<r,\|v\|<r\}$.}

Let's state the meaning of $\epsilon$-boundary for a set in the following definition.
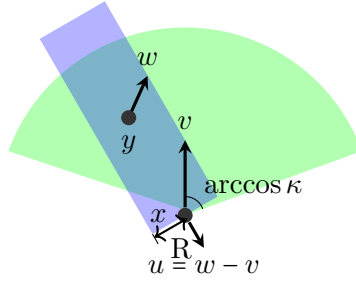
\begin{figure}
\centering
\begin{tikzpicture}[par/.style={circle,inner sep=0,minimum size=2mm,fill=black!80},
ppar/.style={par,fill=red!80},scale=0.5]
\node (x) at (0,0) [label=left:$x$, par] {};

\draw [->,>=stealth,very thick] (x) -- +(90:2) node (v) [above]{$v$};
\draw [->,>=stealth,very thick] (x) -- +(-60:1) node (u) [below]{$u=w-v$};
\node (w) at ($(v)+(u)$) []{};
\draw [->,>=stealth,very thick] ($(x)+(120:3)$) -- ($(x)+(120:3)+(w)$) node [above]{$w$};
\node at ($(x)+(120:3)$) [label=below:$y$, par]{};
\draw [thin] (x)++(20:0.5) arc (20:90:0.5);
\path (x)++(20:0.5) --node [above right]{$\arccos \kappa$}  (90:0.5);
\draw [<->, thick] (x)+(-60:0.1) --node[below =1pt,xshift=4pt,yshift=1pt]{R} ($(x)+(30:-1)+(-60:0.1)$);
\draw [thin] ($(x)+(30:-1)$)--($(x)+(30:-1)+(-60:0.3)$);

\begin{scope}[on background layer]
\fill [green!50, opacity=0.6] (x) -- ++(20:5) arc (20:160:5) -- cycle;
\fill [blue, opacity=0.3, rotate=30] ($(x)-(1,0)$) rectangle ($(x)+(1,6)$);
\end{scope}
\end{tikzpicture}
\caption{\label{fig:Kco}The sketch of the collision-interaction set $\mathcal{K}^{Co}(v,w)=I^{Co}(w-v)\cap \mathcal{C}(v)$. The section is ``vision cone'' $\mathcal{C}(v)$, and the rectangle is the set $I^{Co}(w-v)$.}
\end{figure}

\begin{definition}
Let $K\subset \RR^d$ be a non-empty compact set and $\epsilon>0$. We define the $\epsilon$-boundary of $K$ by:
\begin{equation*}
\partial^\epsilon K:=\partial K+\overline{B(0,\epsilon)}=\{x+y : x\in\partial K,\|y\|\leq \epsilon\}.
\end{equation*}
\end{definition}

From this definition of the $\epsilon$-boundary for a set, we stated the following lemma to characterize its measure. 
\begin{lemma}\label{le:pkco}
For any $(x,v)\in\RR^2\times\RR^2$ and $\epsilon\in\RR$, there exists a constant $C>0$, such that the measure satisfies $|\partial^\epsilon\tilde{\mathcal{K}}^{Co}{(x,v)}|<C\epsilon$. Moreover for any $\Delta x\in\RR^2$, we have $\tilde{\mathcal{K}}^{Co}{(x+\Delta x,v)}=\tilde{\mathcal{K}}^{Co}{(x,v)} + \Delta x$.
\end{lemma}
\begin{proof}
Fix {the velocity} $w\in\RR^2$, we have $\partial\{y\in\RR^{2}:(y-x)\in\mathcal{K}^{Co}{(v,w)}\}=\partial\mathcal{K}^{Co}{(v,w)}+x$. Since $\partial\mathcal{K}^{Co}{(v,w)}$ is a convex polygon restricted in the sector of ``vision cone'' $\mathcal{C}(v)=\{z\in\RR^2:\cos(\alpha(z,v))<\kappa, \|w\|<r\}$, the measure of $\partial\mathcal{K}^{Co}{(v,w)}$ is bounded (see Figure~\ref{fig:Kco}). Since $f$ is compactly support in phase space, we conclude that the measure of $\partial\tilde{\mathcal{K}}^{Co}{(x,v)}=\underset{\|w\|<r}{\bigcup} \partial\mathcal{K}^{Co}{(v,w)}+x$ is bounded. So there exists $C>0$, such that $|\partial^\epsilon\tilde{\mathcal{K}}^{Co}{(x,v)}|<C\epsilon$.
And the second conclusion is easy to prove thanks to the relation
\begin{equation*}
\begin{split}
\tilde{\mathcal{K}}^{Co}{(x+\Delta x,v)} &= \{(y,w)\in\RR^2\times\RR^{2} : (y-x-\Delta x)\in\mathcal{K}^{Co}{(v,w)}\}\\
&=\{(y,w)\in\RR^2\times\RR^2 : (y-x)\in\mathcal{K}^{Co}{(v,w)}\}+\Delta x\\
&=\tilde{\mathcal{K}}^{Co}{(x,v)} + \Delta x.
\end{split}
\end{equation*}
\end{proof}

\begin{lemma}\label{le:lambda}
For any $(x,v),(\tilde x,\tilde v)\in\RR^2\times\RR^2$, we have
\begin{equation*}
\begin{split}
|\lambda^{Co}(t,x,v)-\lambda^{Co}(t,\tilde x,\tilde v)|&
\leq\iint\limits_{\RR^2\times\RR^2} \mathbbm{1}_{\tilde{\mathcal{K}}^{Co}(x,v)\Delta \tilde{\mathcal{K}}^{Co}(\tilde x,\tilde v)}(y,w)f(t,y,w)\rd y\rd w\\
 &\leq C(1+\frac{1}{\|v\|}) \left\|\begin{pmatrix}x\\v\end{pmatrix}-\begin{pmatrix}\tilde x\\ \tilde v\end{pmatrix}\right\|.
\end{split}
\end{equation*}
\end{lemma}

\begin{proof}
\begin{equation}\label{eq:lambdaCo}
\begin{array}{r l}
 &\displaystyle|\lambda^{Co}(t,x,v)-\lambda^{Co}(t,\tilde x,\tilde v)|\\
=&\displaystyle|\iint\limits_{\RR^2\times\RR^2} \mathbbm{1}_{\tilde{\mathcal{K}}^{Co}(x,v)}(y,w)f(t,y,w)\rd y\rd w-\iint\limits_{\RR^2\times\RR^2} \mathbbm{1}_{\tilde{\mathcal{K}}^{Co}(\tilde x,\tilde v)}(y,w)f(t,y,w)\rd y\rd w|\\
=&\displaystyle|\iint\limits_{\RR^2\times\RR^2} (\mathbbm{1}_{\tilde{\mathcal{K}}^{Co}(x,v)}(y,w)-\mathbbm{1}_{\tilde{\mathcal{K}}^{Co}(\tilde x,\tilde v)}(y,w))f(t,y,w)\rd y\rd w|\\
\leq&\displaystyle\iint\limits_{\RR^2\times\RR^2} |\mathbbm{1}_{\tilde{\mathcal{K}}^{Co}(x,v)}(y,w)-\mathbbm{1}_{\tilde{\mathcal{K}}^{Co}(\tilde x,\tilde v)}(y,w)|f(t,y,w)\rd y\rd w\\
\leq&\displaystyle\iint\limits_{\RR^2\times\RR^2} |\mathbbm{1}_{\tilde{\mathcal{K}}^{Co}(x,v)}(y,w)-\mathbbm{1}_{\tilde{\mathcal{K}}^{Co}(\tilde x, v)}(y,w)|f(t,y,w)\rd y\rd w\\
 &+\displaystyle\iint\limits_{\RR^2\times\RR^2} |\mathbbm{1}_{\tilde{\mathcal{K}}^{Co}(\tilde x,v)}(y,w)-\mathbbm{1}_{\tilde{\mathcal{K}}^{Co}(\tilde x,\tilde v)}(y,w)|f(t,y,w)\rd y\rd w\\
:= & I_1+I_2.
\end{array}
\end{equation}

From the Lemma~\ref{le:pkco}, the first term $I_1$ of Inequality~(\ref{eq:lambdaCo}) can be estimated as follows:
\begin{equation*}
\begin{array}{r l}
 &\displaystyle\iint\limits_{\RR^2\times\RR^2} |\mathbbm{1}_{\tilde{\mathcal{K}}^{Co}(x,v)}(y,w)-\mathbbm{1}_{\tilde{\mathcal{K}}^{Co}(\tilde x, v)}(y,w)|f(t,y,w)\rd y\rd w\\
 =&\displaystyle\iint\limits_{\RR^2\times\RR^2} \mathbbm{1}_{\tilde{\mathcal{K}}^{Co}(x,v)\Delta \tilde{\mathcal{K}}^{Co}(\tilde x, v)}(y,w)f(t,y,w)\rd y\rd w\\
 =&\displaystyle\iint\limits_{\RR^2\times\RR^2} \mathbbm{1}_{\partial^{|x-\tilde x|}\tilde{\mathcal{K}}^{Co}(x,v)}(y,w)f(t,y,w)\rd y\rd w\\
\leq&\displaystyle\|f\|_{L^\infty} |\partial^{|x-\tilde x|}\tilde{\mathcal{K}}^{Co}(x,v)|\leq C\|f\|_{L^\infty}\|x-\tilde x\|.
\end{array}
\end{equation*}

Next we estimate the second term, $I_2$, of Inequality~(\ref{eq:lambdaCo}). Since $\tilde{\mathcal{K}}^{Co}(\tilde x,v)\subset \mathcal{C}(v)+\tilde x$ and $\tilde{\mathcal{K}}^{Co}(\tilde x,\tilde v)\subset \mathcal{C}(\tilde v)+\tilde x$, we only need to consider $y\in\mathcal{C}(v)\bigcup\mathcal{C}(\tilde v)+\tilde x=(\mathcal{C}(v)\Delta\mathcal{C}(\tilde v)+\tilde x)\bigcup(\mathcal{C}(v)\bigcap\mathcal{C}(\tilde v)+\tilde x)$.\\
So,
\begin{itemize}
\item[(i)] in one hand, we have estimations as follows
\begin{equation*}
|\mathcal{C}(v)\Delta\mathcal{C}(\tilde v)+\tilde x|=|\mathcal{C}(v)\Delta\mathcal{C}(\tilde v)|=2\cdot \frac{1}{2} r^2 \cdot|\alpha(v,\tilde v)|.  
\end{equation*}

When $\|v-\tilde v\|< {1\over 2}\|v\|$, we have $|\alpha(v,\tilde v)|<\arcsin \cfrac{\|v-\tilde v\|}{\|v\|}\leq C' \cfrac{\|v-\tilde v\|}{\|v\|}$ for some proper constant $C'>2\pi$ independent of $\|v\|$. And if $\|v-\tilde v\|\geq {1\over 2}\|v\|$, we have $|\alpha(v,\tilde v)|\leq \pi \leq C' \cfrac{\|v-\tilde v\|}{\|v\|}$.

Then we deduce that
\begin{equation*}
|\mathcal{C}(v)\Delta\mathcal{C}(\tilde v)+\tilde x|=2\cdot \frac{1}{2} r^2 \cdot|\alpha(v,\tilde v)|\leq C \cfrac{\|v-\tilde v\|}{\|v\|},
\end{equation*}
where $C$ is a constant.
 
Thus, from the compactly support of $f$ in phase space, we get
\begin{equation*}
\begin{array}{r l}
 &\displaystyle\int\limits_{\mathcal{C}(v)\Delta\mathcal{C}(\tilde v)+\tilde x}\hskip-4mm \rd y\int\limits_{\RR^2} \mathbbm{1}_{\tilde{\mathcal{K}}^{Co}(\tilde x,v)\Delta \tilde{\mathcal{K}}^{Co}(\tilde x,\tilde v)}f(t,y,w)\rd w\\
\leq& \hskip-4mm\displaystyle\int\limits_{\mathcal{C}(v)\Delta\mathcal{C}(\tilde v)}\hskip-4mm \rd y \int\limits_{\RR^2} \|f\|_{L^\infty} \rd w\leq C \cfrac{\|v-\tilde v\|}{\|v\|} \|f\|_{L^\infty} \cdot \pi r^2.
\end{array}
\end{equation*}

%
%

\item[(ii)] In other hand, let $S(y,v):=\{w\in\RR^2:\mathbbm{1}_{\tilde{\mathcal{K}}^{Co}(\tilde x,v)}(y, w)=1, \|w\|\leq r\}$ be the intersection of the sector $\{w:\mathbbm{1}_{\tilde{\mathcal{K}}^{Co}(\tilde x,v)}(y, w)=1\}$ and the disc $\{w:\|w\|\leq r\}$. And it is easy to see that $S(y,\tilde v) = S(y,v) + (v-\tilde v)$, so $S(y,v)\Delta S(y,\tilde v) \subset \partial^{\|v-\tilde v\|} S(y,v)$. Then we have $|S(y,v)\Delta S(y,\tilde v)|<C\|v-\tilde v\|$. And we can obtain the estimation
\begin{equation*}
\begin{array}{r l}
 &\displaystyle\int\limits_{\mathcal{C}(v)\bigcap\mathcal{C}(\tilde v)+\tilde x}\hskip-4mm\rd y\int\limits_{\RR^2} \mathbbm{1}_{\tilde{\mathcal{K}}^{Co}(\tilde x,v)\Delta \tilde{\mathcal{K}}^{Co}(\tilde x,\tilde v)}f(t,y,w)\rd w\\
\leq&\hskip-4mm\displaystyle\int\limits_{\mathcal{C}(v)\bigcap\mathcal{C}(\tilde v)+\tilde x}\hskip-4mm\rd y \int\limits_{\RR^2} \mathbbm{1}_{S(y,v)\Delta S(y,\tilde v)} \|f\|_{L^\infty}\rd w\leq \pi r^2 \cdot \|f\|_{L^\infty} C \|v-\tilde v\|.
\end{array}
\end{equation*}
\end{itemize}

From the points (i) and (ii) we deduce the estimate of $I_2$ as follows
\begin{equation*}
\begin{array}{r l}
 &\displaystyle\iint\limits_{\RR^2\times\RR^2} |\mathbbm{1}_{\tilde{\mathcal{K}}^{Co}(\tilde x,v)}(y,w)-\mathbbm{1}_{\tilde{\mathcal{K}}^{Co}(\tilde x,\tilde v)}(y,w)|f(t,y,w)\rd y\rd w\\
=&\displaystyle\iint\limits_{\RR^2\times\RR^2} \mathbbm{1}_{\tilde{\mathcal{K}}^{Co}(\tilde x,v)\Delta \tilde{\mathcal{K}}^{Co}(\tilde x,\tilde v)}f(t,y,w)\rd y\rd w\\
=&\displaystyle(\int\limits_{\mathcal{C}(v)\Delta\mathcal{C}(\tilde v)+\tilde x}+\int\limits_{\mathcal{C}(v)\bigcap\mathcal{C}(\tilde v)+\tilde x})\rd y\int\limits_{\RR^2} \mathbbm{1}_{\tilde{\mathcal{K}}^{Co}(\tilde x,v)\Delta \tilde{\mathcal{K}}^{Co}(\tilde x,\tilde v)}f(t,y,w)\rd w\\
\leq& \displaystyle\int\limits_{\mathcal{C}(v)\Delta\mathcal{C}(\tilde v)} \rd y \int\limits_{\RR^2} \|f\|_{L^\infty} \rd w + \int\limits_{\mathcal{C}(v)\bigcap\mathcal{C}(\tilde v)+\tilde x}\rd y \int\limits_{\RR^2} \mathbbm{1}_{S(y,v)\Delta S(y,\tilde v)} \|f\|_{L^\infty}\rd w\\
\leq& C \cfrac{\|v-\tilde v\|}{\|v\|} \|f\|_{L^\infty} \cdot \pi r^2+\pi r^2 \cdot \|f\|_{L^\infty} C \|v-\tilde v\|\leq  C(1+\cfrac{1}{\|v\|})\|v-\tilde v\|.
\end{array}
\end{equation*}

Finally, we achieve the consequence of the Lemma,
\begin{equation*}
|\lambda^{Co}(t,x,v)-\lambda^{Co}(t,\tilde x,\tilde v)|
\leq  C \|x-\tilde x\| + C(1+\cfrac{1}{\|v\|})\|v-\tilde v\|\leq C(1+\frac{1}{\|v\|}) \left\|\begin{pmatrix}x\\v\end{pmatrix}-\begin{pmatrix}\tilde x\\ \tilde v\end{pmatrix}\right\|.
\end{equation*}
\end{proof}

Now we are ready to investigate as follows the proof of Proposition~\ref{pro:co}
\begin{proof}[Proof of Proposition~\ref{pro:co}]

The first conclusion is easy to prove. Since $m(z,v,w)$ is bounded, and $\lambda^{Co}(t,x,v)\geq \beta$, we have
\begin{equation*}
|\mathcal{F}^{Co}(f)(t,x,v)| \leq \frac{1}{\beta} C \Big| \iint\limits_{\RR^2\times\RR^2} f(y,w)\rd y\rd w \Big|\cdot\|v\| \leq C \|f\|_{L^1}\|v\|.
\end{equation*}
Next, we start to estimate $|\mathcal{F}^{Co}(f)(t,x,v)-\mathcal{F}^{Co}(f)(t,\tilde x,\tilde v)|$ by writing
\begin{equation*}
\begin{split}
&|\mathcal{F}^{Co}(f)(t,x,v)-\mathcal{F}^{Co}(f)(t,\tilde x,\tilde v)|\\
=&\left|\cfrac{1}{\lambda^{Co}(t,x,v)} \iint\limits_{\RR^2\times\RR^2} m^{Co}(y-x,v,w)\mathbbm{1}_{\tilde{\mathcal{K}}^{Co}(x,v)}(y,w)f(t,y,w)\rd y\rd w \cdot v^{\bot} \right.\\
&- \left.\cfrac{1}{\lambda^{Co}(t,\tilde x,\tilde v)} \iint\limits_{\RR^2\times\RR^2} m^{Co}(y-\tilde x,\tilde v,w)\mathbbm{1}_{\tilde{\mathcal{K}}^{Co}(\tilde x,\tilde v)}(y,w)f(t,y,w)\rd y\rd w\cdot \tilde{v}^\bot\right| \\
\leq&\left|(\cfrac{1}{\lambda^{Co}(t,x,v)}-\cfrac{1}{\lambda^{Co}(t,\tilde x,\tilde v)}) \iint\limits_{\RR^2\times\RR^2} m^{Co}(y-x,v,w)\mathbbm{1}_{\tilde{\mathcal{K}}^{Co}(x,v)}(y,w)f(t,y,w)\rd y\rd w \cdot v^{\bot}\right|\\
&+\left|\cfrac{1}{\lambda^{Co}(t,\tilde x,\tilde v)}(\iint\limits_{\RR^2\times\RR^2} m^{Co}(y-x,v,w)\mathbbm{1}_{\tilde{\mathcal{K}}^{Co}(x,v)}(y,w)f(t,y,w)\rd y\rd w \cdot v^{\bot}\right.\\
&-\left.\iint\limits_{\RR^2\times\RR^2} m^{Co}(y-\tilde x,\tilde v,w)\mathbbm{1}_{\tilde{\mathcal{K}}^{Co}(\tilde x,\tilde v)}(y,w)f(t,y,w)\rd y\rd w \cdot \tilde{v}^{\bot})\right|\\
&:= I_1 + I_2.
\end{split}
\end{equation*}

From Lemma~\ref{le:lambda}, we have
\begin{equation*}
\begin{split}
I_1 &\leq C\|f\|_{L^1}\|v\| \cfrac{|\lambda^{Co}(t,x,v)-\lambda^{Co}(t,\tilde x,\tilde v)|}{\lambda^{Co}(t,x,v)\cdot\lambda^{Co}(t,\tilde x,\tilde v)}\leq C\|f\|_{L^1}(1+\|v\|) \frac{1}{\beta^2} \|(x,v)-(\tilde x,\tilde v)\|\\
& \leq C\|f\|_{L^1}(1+\|v\|) \left\|\begin{pmatrix}x\\v\end{pmatrix}-\begin{pmatrix}\tilde x\\ \tilde v\end{pmatrix}\right\|.
\end{split}
\end{equation*}

For the estimation of $I_2$, we split it into two parts, thanks to the given sets $S_1=\tilde{\mathcal{K}}^{Co}(x,v)\cap \tilde{\mathcal{K}}^{Co}(\tilde x,\tilde v)$, $S_2=\tilde{\mathcal{K}}^{Co}(x,v)\Delta \tilde{\mathcal{K}}^{Co}(\tilde x,\tilde v)$.
\begin{equation*}
\begin{array}{l l}
I_2 \leq& \left|\cfrac{1}{\lambda^{Co}(t,\tilde x,\tilde v)}(\iint\limits_{S_1} m^{Co}(y-x,v,w)f(t,y,w)\rd y\rd w \cdot v^{\bot}\right.-\left.\iint\limits_{S_1} m^{Co}(y-\tilde x,\tilde v,w)f(t,y,w)\rd y\rd w \cdot \tilde{v}^{\bot})\right|\\
&+ \left|\cfrac{1}{\lambda^{Co}(t,\tilde x,\tilde v)}(\iint\limits_{S_2} m^{Co}(y-x,v,w)\mathbbm{1}_{\tilde{\mathcal{K}}^{Co}(x,v)}(y,w)f(t,y,w)\rd y\rd w \cdot v^{\bot}\right.\\
&\left.-\iint\limits_{S_2} m^{Co}(y-\tilde x,\tilde v,w)\mathbbm{1}_{\tilde{\mathcal{K}}^{Co}(\tilde x,\tilde v)}(y,w)f(t,y,w)\rd y\rd w \cdot \tilde{v}^{\bot})\right|\\
&:= I_{21} + I_{22}.
\end{array}
\end{equation*}

From the smoothness of $\cos^\varepsilon(\alpha(z,w)), e^{-\tau(z,w-v)}$ and $g(\cfrac{(w-v)\times z}{\|z\|^2})$, we know that $m(z,v,w)$ is Lipschitz continuous. Then
\begin{equation*}
I_{21} \leq \frac{1}{\beta}\|f\|_{L^1}\|v\|C \left\|\begin{pmatrix}x\\v\end{pmatrix}-\begin{pmatrix}\tilde x\\ \tilde v\end{pmatrix}\right\|.
\end{equation*}

From the boundedness of $m^{Co}$ and the definition of $S_2$, we have
\begin{equation*}
\begin{array}{r c l}
I_{22} &\leq& \left|\cfrac{1}{\lambda^{Co}(t,\tilde x,\tilde v)}\iint\limits_{S_2} m^{Co}(y-x,v,w)\mathbbm{1}_{\tilde{\mathcal{K}}^{Co}(x,v)}(y,w)f(t,y,w)\rd y\rd w \cdot v^{\bot}\right|\\
& &+ \left|\cfrac{1}{\lambda^{Co}(t,\tilde x,\tilde v)}\iint\limits_{S_2} m^{Co}(y-\tilde x,\tilde v,w)\mathbbm{1}_{\tilde{\mathcal{K}}^{Co}(\tilde x,\tilde v)}(y,w)f(t,y,w)\rd y\rd w \cdot \tilde{v}^{\bot}\right|\\
& \leq& \left|\cfrac{1}{\lambda^{Co}(t,\tilde x,\tilde v)}\iint\limits_{S_2} m^{Co}(y-x,v,w)\mathbbm{1}_{\tilde{\mathcal{K}}^{Co}(x,v)}(y,w)f(t,y,w)\rd y\rd w \right|\cdot \|v\|\\
& &+ \left|\cfrac{1}{\lambda^{Co}(t,\tilde x,\tilde v)}\iint\limits_{S_2} m^{Co}(y-\tilde x,\tilde v,w)\mathbbm{1}_{\tilde{\mathcal{K}}^{Co}(\tilde x,\tilde v)}(y,w)f(t,y,w)\rd y\rd w\right|\cdot \|v\|\\
& &+ \left|\cfrac{1}{\lambda^{Co}(t,\tilde x,\tilde v)}\iint\limits_{S_2} m^{Co}(y-\tilde x,\tilde v,w)\mathbbm{1}_{\tilde{\mathcal{K}}^{Co}(\tilde x,\tilde v)}(y,w)f(t,y,w)\rd y\rd w \right|\cdot \|v-\tilde v\|\\
&\leq & \cfrac{2}{\beta} C \|f\|_{L^\infty} \|v\| |S_2| +  \cfrac{1}{\beta} C \|f\|_{L^1} \|v-\tilde v\| \quad {\text{where }|S_2|\text{ is the measure of } S_2}\\
& \leq & C\|v\|(1+\cfrac{1}{\|v\|})\left\|\begin{pmatrix}x\\v\end{pmatrix}-\begin{pmatrix}\tilde x\\ \tilde v\end{pmatrix}\right\|+C\|v-\tilde v\|\\
& \leq & C(1+\|v\|)\left\|\begin{pmatrix}x\\v\end{pmatrix}-\begin{pmatrix}\tilde x\\ \tilde v\end{pmatrix}\right\|.
\end{array}
\end{equation*}

The second last inequality above is deduced from the proof of Lemma~\ref{le:lambda}. Thus, we have
\begin{equation*}
\begin{split}
I_2 &\leq C(1+\|v\|)\left\|\begin{pmatrix}x\\v\end{pmatrix}-\begin{pmatrix}\tilde x\\ \tilde v\end{pmatrix}\right\|.
\end{split}
\end{equation*}

Finally, we get
\begin{equation*}
|\mathcal{F}^{Co}(f)(t,x,v)-\mathcal{F}^{Co}(f)(t,\tilde x,\tilde v)|\leq C(1+\|v\|)\left\|\begin{pmatrix}x\\v\end{pmatrix}-\begin{pmatrix}\tilde x\\ \tilde v\end{pmatrix}\right\|.
\end{equation*}
\end{proof}

We have similar conclusions for $\mathcal{F}^{Im}(f), \mathcal{F}^{Fo}(f)$, and the proof follows the same strategy as that for $\mathcal{F}^{Co}$. Otherwise, one can easily complete the proof of the following proposition that we leave to the reader.

\begin{proposition}\label{pro:imfo}
For any $(x,v)\in\RR^2\times\RR^2, t\in[0,T]$, we have following estimations,
\begin{align*}
|\mathcal{F}^{Im}(f)(t,x,v)| \leq C\|v\| \|f\|_{L^1},\\
|\mathcal{F}^{Fo}(f)(t,x,v)| \leq C\|v\| \|f\|_{L^1},
\end{align*} and
\begin{align*}
|\mathcal{F}^{Im}(f)(t,x,v)-\mathcal{F}^{Im}(f)(t,\tilde x,\tilde v)|\leq C(1+\|v\|)\left\|\begin{pmatrix}x\\v\end{pmatrix}-\begin{pmatrix}\tilde x\\ \tilde v\end{pmatrix}\right\|,\\
|\mathcal{F}^{Fo}(f)(t,x,v)-\mathcal{F}^{Fo}(f)(t,\tilde x,\tilde v)|\leq C(1+\|v\|)\left\|\begin{pmatrix}x\\v\end{pmatrix}-\begin{pmatrix}\tilde x\\ \tilde v\end{pmatrix}\right\|,
\end{align*}
where
\begin{equation*}
\begin{array}{r l}
\mathcal{F}^{Im}(f)(t,x,v) &= -\Omega_{Im}\cdot v\\
 &= \cfrac{1}{\lambda^{Im}(t,x,v)} \iint\limits_{\RR^2\times\RR^2} m^{Im}(y-x,v,w)\mathbbm{1}_{\mathcal{K}^{Im}(v,w)}(y-x)f(t,y,w)\rd y\rd w \cdot (-v),\\

\mathcal{F}^{Fo}(f)(t,x,v) &= \Omega_{Fo}\cdot v^\bot\\
 &= \cfrac{1}{\lambda^{Fo}(t,x,v)} \iint\limits_{\RR^2\times\RR^2} m^{Fo}(y-x,v,w)\mathbbm{1}_{\mathcal{K}^{Fo}(v,w)}(y-x)f(t,y,w)\rd y\rd w \cdot v^\bot.
\end{array}
\end{equation*}
\end{proposition}

For the estimation of $|\mathcal{F}(f)(t,x,v)-\mathcal{F}(f)(t,\tilde x,\tilde v)|$, we write
\begin{equation*}
\begin{split}
|\mathcal{F}(f)(t,x,v)-\mathcal{F}(f)(t,\tilde x,\tilde v)|\leq
|\mathcal{F}^{Co}(f)(t,x,v)-\mathcal{F}^{Co}(f)(t,\tilde x,\tilde v)|\\
+|\mathcal{F}^{Im}(f)(t,x,v)-\mathcal{F}^{Im}(f)(t,\tilde x,\tilde v)|\\
+|\mathcal{F}^{Fo}(f)(t,x,v)-\mathcal{F}^{Fo}(f)(t,\tilde x,\tilde v)|\\
+|\mathcal{F}^{Ex}(f)(t,x,v)-\mathcal{F}^{Ex}(f)(t,\tilde x,\tilde v)|.
\end{split}
\end{equation*}
 By Proposition~\ref{pro:co} and~\ref{pro:imfo}, we know $\mathcal{F}^{Co}(f)(t,x,v),\mathcal{F}^{Im}(f)(t,x,v)$ and $\mathcal{F}^{Fo}(f)(t,x,v)$ are all Lipschitz continuous. From the assumption that $\nabla_x V(x)$ is Lipschitz continuous, we know that the exit-interaction force $\mathcal{F}^{Ex}(f)$ is also Lipschitz continuous,
\begin{equation*}
\begin{array}{r l}
 &|\mathcal{F}^{Ex}(f)(t,x,v)-\mathcal{F}^{Ex}(f)(t,\tilde x,\tilde v)|\\
\leq&\|\nabla_x V(x)-\nabla_x V(\tilde x)\|+\sigma\|v-\tilde v\|\\
\leq& C\|x-\tilde x\|+\sigma\|v-\tilde v\|\leq C\left\|\begin{pmatrix}x\\v\end{pmatrix}-\begin{pmatrix}\tilde x\\ \tilde v\end{pmatrix}\right\|.
\end{array}
\end{equation*}

Finally, we achieve the proof of Theorem~\ref{th:lip},
\begin{equation*}
|\mathcal{F}(f)(t,x,v)-\mathcal{F}(f)(t,\tilde x,\tilde v)|\leq C(1+\|v\|)\left\|\begin{pmatrix}x\\v\end{pmatrix}-\begin{pmatrix}\tilde x\\ \tilde v\end{pmatrix}\right\|,
\end{equation*}
for all $x,v,\tilde x,\tilde v\in\RR^2$ and $t\in[0,T]$.


\bibliographystyle{SIAM}
\bibliography{bib}

\end{document}